\documentclass{article}
\usepackage{amsmath,amssymb,amsthm}

\theoremstyle{definition}
\newtheorem{theorem}{Theorem}[section]
\newtheorem{lemma}[theorem]{Lemma}

\newtheorem{remark}[theorem]{Remark}
\newtheorem{definition}[theorem]{Definition}

\newtheorem{note}[theorem]{Note}
\newtheorem{proposition}[theorem]{Proposition}

\makeatletter

\@addtoreset{equation}{section}
\makeatother

\title{Harmonic manifolds of hypergeometric type and spherical Fourier transform}

\author{Mitsuhiro Itoh
\footnote{University of Tsukuba,
1-1-1 Tennodai, Tsukuba-shi, Ibaraki 305-8577, JAPAN\quad
e-mail : itohm@math.tsukuba.ac.jp}
and Hiroyasu Satoh
\footnote{Liberal Arts and Sciences, Nippon Institute of Technology, 4-1 Gakuendai, Miyashiro-machi, Minamisaitama-gun, Saitama 345-8501 JAPAN\quad
e-mail : hiroyasu@nit.ac.jp}}

\date{\today}

\begin{document}

\maketitle

\begin{abstract}
The spherical Fourier transform on a harmonic Hadamard manifold $(X^n,g)$ of positive volume entropy is studied. If $(X^n,g)$ is of hypergeometric type, namely spherical functions of $X$ are represented by the Gauss hypergeometric functions, the inversion formula, the convolution rule together with the Plancherel theorem are shown by the representation of the spherical functions in terms of the Gauss hypergeometric functions. A geometric characterization of hypergeometric type is derived in terms of volume density of geodesic spheres. Geometric properties of $(X^n,g)$ are also discussed.
\end{abstract}

\section{Introduction and Main Results}

Let $(X^n,g)$ be an $n$-dimensional harmonic manifold.
A Riemannian manifold is called harmonic, if it is complete and admits at any point $x\in X$ a nontrivial radial harmonic function around $x$.
A real valued function $f$ on $X$ is radial with respect to a point $o$, when $f$ is written by $f(x)= {\tilde{f}}(r(x))$, $x\in X$, where $\tilde{f} : [0,+\infty)\rightarrow {\Bbb R}$ and $r(x)=d(x,o)$ is the distance from $o$.
A harmonic manifold can be then characterized in the following theorem.

\begin{theorem}[{cf. \cite[6.21 Proposition]{Besse}, \cite[Lemma 1.1]{Sz}}]\label{equiv_harmonicmfd}\rm Let $(X,g)$ be a complete Riemannian manifold.
Let $S(y,r)$ be a geodesic sphere of center $y\in X$ and radius $r$. The following are equivalent each other.
\begin{enumerate}
\item $(X,g)$ is harmonic.
\item the mean curvature $\sigma_y(\exp_y ru)$ of $S(y;r)$ is a radial function, i.e., $\sigma_y$ does not depend on $u\in S_yX$, where $S_yX$ is the space of unit tangent vectors at $y$.
\item the volume density $\Theta_y(\exp_y ru)$ of $S(y;r)$ is a radial function, i.e., $\Theta_y$ does not depend on $u\in S_yX$.
\item the averaging operator ${\mathcal{M}}_y$ commutes with the Laplace-Beltrami operator $\Delta$;\, $\Delta{\mathcal{M}}_y = {\mathcal{M}}_y\Delta$.\, 
For a smooth function $f$, ${\mathcal{M}}_y(f)$ is a smooth radial function on $X$ whose value is the average of $f$ on $S(y;r)$;
\begin{equation}\label{avrgop}
{\mathcal{M}}_y(f)(r) := \frac{1}{\int_{S(y;r)}dv_{S(y;r)}}\, \int_{x\in S(y;r)} f(x)\,d v_{S(y;r)}.
\end{equation}
\end{enumerate}
\end{theorem}

From \cite{Besse} there exists a function $\Theta(r)$ of $r$ such that $\Theta_y(\exp_y r u)=\Theta(r)$, $r \geq 0$ and
hence $\sigma_y(\exp_y ru) = \sigma(r)$, $r>0$ for any $u\in S_yX$, $y\in X$, where $\sigma(r):= (\log \Theta(r))'$.
 
\medskip

The Euclidean spaces and the rank one symmetric spaces are typical harmonic manifolds, suggesting the local symmetricity conjecture posed by Lichnerowicz \cite{Lich}.
The conjecture is affirmative in the compact case from the results by \cite{BCG, Sz}.
However, there exist non-symmetric harmonic manifolds among the class of Damek-Ricci spaces.
A Damek-Ricci space is a one-dimensional extension of a generalized Heisenberg group equipped with a left invariant Riemannian metric \cite{BTV}.
It is known that each Damek-Ricci space is harmonic, Hadamard and of positive volume entropy \cite{ADY, BTV}. 

In harmonic analysis, the role played by radial functions is crucially important. Spherical Fourier analysis, i.e., Fourier analysis of radial functions on a harmonic manifold is an interesting subject of harmonic analysis. 
The inversion formula, the Plancherel theorem and the theorem of Paley-Wiener's type are major subjects of the spherical Fourier analysis on a harmonic manifold, as illustrated in \cite{H}.
The spherical Fourier analysis on Damek-Ricci spaces have been studied by several authors \cite{Ri, DR, ADY, R}, following the work of Helgason \cite{H}.
One can develop the theory of spherical Fourier transform on a harmonic Hadamard manifold similar to the Damek-Ricci space case, assuming that the spherical functions become Gauss hypergeometric functions by a certain variable change. Most of the properties satisfied by the spherical functions on a Damek-Ricci space remain true even on a harmonic Hadamard manifold of hypergeometric type.

The spherical functions $\varphi_{\lambda}(r)$ on a harmonic Hadamard manifold $(X^n,g)$ are the eigenfunctions of the radial part 
\begin{eqnarray}\label{radiallaplacian}
\Delta^\mathrm{rad}=-\left(\frac{d^2}{dr^2}+ \sigma(r)\frac{d}{dr}\right) = - \frac{1}{\Theta(r)}\left(\frac{d}{dr}\left(\Theta(r)\frac{d}{dr}\right)\right)
\end{eqnarray}
of the Laplace-Beltrami operator $\Delta$ satisfying $\varphi_{\lambda}(0)=1$, $\varphi_{\lambda}'(0)=0$. Refer to \cite{H, ADY, Sz}.
Here we denote by $\sigma(r)$ and $\Theta(r)$ the mean curvature and the volume density of $S(o,r)$ of center $o$ and radius $r$, respectively, for which it holds $\displaystyle{\sigma(r) = d/dr\, \log \Theta(r)}$.
The spherical functions provide the spherical Fourier transform $\mathcal{H}$ for a smooth radial function $f=f(r)$, $r(x)=d(o,x)$, $x\in X$, of compact support on $X$ as
\begin{definition}\label{definition-1}
$\mathcal{H} : f=\ f(r) \mapsto \hat{f} = \hat{f}(\lambda)$;
\begin{equation}\label{sphericalfourier}
\hat{f}(\lambda)
= \int_X f(r(x) \varphi_{\lambda}(r(x))\,dv_g(x)
=\omega_{n-1} \int_0^{\infty} f(r) \varphi_{\lambda}(r) \Theta(r)\,dr,\quad\lambda\in{\Bbb C},
\end{equation}
where $\omega_{n-1}= {\rm vol}\,S^{n-1}(1)$. \end{definition}
Then $\mathcal{H}$ defines a linear map
\begin{eqnarray}
\mathcal{H}\, : C^{\infty,\mathrm{rad}}_0(X) \rightarrow \mathcal{PW}({\Bbb C})_{\mathrm{even}}\end{eqnarray}
from the space $C^{\infty,\mathrm{rad}}_0(X)$ of smooth radial functions on $X$ of compact support to the space $\mathcal{PW}({\Bbb C})_{\mathrm{even}}$ of even holomorphic functions of $\lambda\in{\Bbb C}$ of exponential type.
For the space $\mathcal{PW}({\Bbb C})_{\mathrm{even}}$ see Definition \ref{paleywiener}. 
 
\medskip

A harmonic Hadamard manifold is said to be of hypergeometric type, if every spherical function $\varphi_{\lambda}$ is converted into a Gauss hypergeometric function $F(a,b,c;z)$, i.e., a smooth solution to the hypergeometric differential equation, via the variable change $\displaystyle{z= - \sinh^2 r/2}$.
Refer to \cite{R} and \cite{ADY} for introducing the variable change on Damek-Ricci spaces. 

Let $(X^n,g)$, $n\geq 3$ be a harmonic Hadamard manifold of volume entropy $Q>0$ and of hypergeometric type.
The aim of this paper is devoted and developing Riemannian geometry of $(X^n,g)$.
 
We show the Plancherel theorem which asserts that $\mathcal{H}$ is isometric with respect to certain inner products by applying the inversion formula which was obtained by the authors in \cite{ItohSatoh-2} together with the convolution rule, valid for simply connected harmonic Hadamard manifolds of $Q>0$.
The convolution rule is shown in Theorem \ref{convolutionrule}. Remark that the injectivity of $\mathcal{H}$ is shown in \cite[Theorem 3.12]{PS} in the sense of distribution. 

\begin{theorem}\label{volumedensmeancurv}
Let $(X^n, g)$ be an $n(\geq 3)$-dimensional harmonic Hadamard manifold of volume entropy $Q>0$.
Assume that $(X^n, g)$ is of hypergeometric type.
Then the volume density $\Theta(r)$ and the mean curvature $\sigma(r)$ of a geodesic sphere $S(o;r)$ of $(X^n, g)$ are described respectively by
\begin{align}
\label{volumedensity-1}
\Theta(r)=& 2^{n-1} \, \sinh^{n-1}\frac{r}{2}\cosh^{(2Q-(n-1))}\frac{r}{2},\\
\label{volumedensity-2}
\sigma(r)=& \frac{n-1}{2}\coth \frac{r}{2}+ (Q-\frac{n-1}{2})\tanh \frac{r}{2}.
\end{align}
\end{theorem} 

Furthermore we obtain geometric characterizations of hypergeometric type.
\begin{theorem}[\cite{ItohSatoh-2}] \label{geomcharacterization}
Let $(X,g)$ be a harmonic Hadamard manifold.
Then the following are equivalent each other.
\begin{enumerate}
\item $(X,g)$ is of hypergeometric type,
\item $\sigma(r)$ fulfills
\begin{eqnarray}
\sigma(r) = c_1 \coth \frac{r}{2} + c_2 \tanh \frac{r}{2}
\end{eqnarray}for some constants $c_1 > 0$ and $c_2$ satisfying $c_1+ c_2 > 0$,
\item the volume density $\Theta(r)$ of $S(o; r)$ is represented by
\begin{eqnarray}\label{volumedensity}
\Theta(r) = k \sinh^{2 c_1} \frac{r}{2}\, \cosh^{2 c_2}\frac{r}{2}
\end{eqnarray}
for a constant $k > 0$ with some constants $c_1 > 0$ and $c_2$ satisfying $c_1+ c_2 > 0$.
\end{enumerate}
\end{theorem}

\begin{remark}
All Damek-Ricci spaces are of hypergeometric type. Refer to \cite{ADY, R}.
In \cite[(1.16)]{ADY} the volume density of a Damek-Ricci space $S$ has the form
\begin{eqnarray*}
\Theta(r) = 2^{m+k} \sinh^{m+k}\frac{r}{2}\, \cosh^k \frac{r}{2}.
\end{eqnarray*}
Here $\dim S= m+k+1$, $m= \dim \mathfrak{v}$, $k = \dim \mathfrak{z}$ where $\mathfrak{v}$ and $\mathfrak{z}$ are the canonical subspaces of the Lie algebra of $S$, respectively and $Q= m/2+k$ so $2Q-(n-1) = k$.
See also \cite{DR}.
\end{remark}

The exact form of the volume density appeared at \eqref{volumedensity-1} implies that $(X^n,g)$ is of purely exponential volume growth in the sense of G. Knieper \cite{K}.  
 
\begin{theorem}\label{exponentialvolumegrowth}\rm
Let $(X^n,g)$ be a harmonic Hadamard manifold of $Q >0$. If $(X^n,g)$ is of hypergeometric type.
Then $(X^n,g)$ is of purely exponential volume growth.
Hence, $(X^n,g)$ is Gromov hyperbolic, rank one and the geodesic flow of the unit sphere bundle is Anosov.
\end{theorem}

A harmonic manifold is Einstein from the Ledger's formula. We define $\delta_g<0$ for $(X^n,g)$ by $\delta_g\, g = {\rm Ric}_g/(n-1)$. Then, the Bishop volume comparison theorem yields the following.

\begin{theorem}[\cite{ItohSatoh-3}]\label{volumeentropy-2}
Let $(X^n,g)$ be a harmonic Hadamard manifold of $Q>0$. Assume $(X^n,g)$ is of hypergeometric type.
Then, $Q$ satisfies
\begin{eqnarray}\label{volumeentropy-1}
\frac{n-1}{3}\left(1+2\vert\delta_g\vert\right) = Q \leq (n-1)\sqrt{\vert\delta_g\vert}.
\end{eqnarray}
Moreover, $\delta_g$ fulfills
\begin{eqnarray}
\frac{1}{4}\leq \vert\delta_g\vert \leq 1
\end{eqnarray}
and hence $Q$ must satisfy 
\begin{eqnarray}\label{qinequality}
\frac{n-1}{2} \leq Q \leq n-1.
\end{eqnarray}
Equality $Q= n-1$ holds if and only if $(X^n,g)$ is isometric to ${\Bbb R}{\rm H}^n(-1)$ and $Q= (n-1)/2$ holds if and only if $(X^n,g)$ is isometric to ${\Bbb R}{\rm H}^n(-1/4)$, where ${\Bbb R}{\rm H}^n(K)$ denotes the real hyperbolic space of constant sectional curvature $K$.
\end{theorem} 

As an additional remark we observe that a real hyperbolic space ${\Bbb R}{\rm H}^n(K)$ is of hypergeometric type if and only if $K = -1$ or $K = -1/4$.

The exact form of $\sigma(r)$ presented in \eqref{volumedensity-2} implies the following curvature inequalities, by employing the Ledger's formulas together with \eqref{sigmaexpansion} and \eqref{sigmaexpansioncoefficients}.
\begin{theorem}\label{curvatureinequalities}
Let $(X^n,g)$ be a harmonic Hadamard manifold of $Q>0$ and of hypergeometric type.
The following provides a relation between the curvature invariants, the Ricci curvature ${\rm Ric}_u={\rm Tr}\, R_u$ and ${\rm Tr}\, R_u^{\, 2}$, where $R_u$, $u\in S_yX$ is the Jacobi operator; 
\begin{eqnarray}\label{riccitrace}
{\rm Tr}\, R_u^{\, 2} = -\frac{1}{2^2}\, \left((n-1)+ 5 {\rm Ric}_u\right).
\end{eqnarray}
Further the inequalities 
\begin{eqnarray}\label{inequalities}
 -(n-1) \leq {\rm Ric}_u \leq - \frac{n-1}{2^2},\ \ n-1 \geq {\rm Tr}\, R_u^{\, 2} \geq \frac{n-1}{2^4}
\end{eqnarray}
hold for any $u\in S_yX$, $y\in X$. 

Moreover, equality ${\rm Ric}_u= -(n-1)$ holds if and only if ${\rm Tr}\, R_u^{\, 2}= n-1$ holds.
Either of them holds if and only if $(X^n,g)$ is isometric to ${\Bbb R}{\rm H}^n(-1)$. Moreover equality ${\rm Ric}_u= -(n-1)/2^2$ holds if and only if ${\rm Tr}\, R_u^{\, 2}= (n-1)/2^4$ holds.
Either of them holds if and only if $(X^n,g)$ is isometric to ${\Bbb R}{\rm H}^n(-1/4)$. 
\end{theorem}

We proved in \cite{ItohSatoh-2} the inversion formula for the spherical Fourier transform by utilizing the Green's formula for the Laplace-Beltrami operator $\Delta$ together with the Riemann-Lebesgue theorem,\, by which G\"otze had verified the inversion formula for Mehler-Fock integral transformation in \cite{G}.

Our inversion formula states that $f\in C^{\infty, \mathrm{rad}}_0(X)$ is recovered by ${\hat f}= {\hat f}(\lambda)$, in terms of the linear map $\mathcal{H}^-$; $f(r)= (\mathcal{H}^-{\hat f})(r)$. 
Here $\mathcal{H}^-$, which we call the dual spherical Fourier transform, is defined for $h=h(\lambda)\in \mathcal{PW}({\Bbb C})_{\mathrm{even}}$ by
\begin{eqnarray}\label{definitiondualFT}
\mathcal{H}^- \, :\, h \mapsto {\breve{h}}(r) = 2\ d_g \int_0^{\infty} h(\lambda) \varphi_{\lambda}(r) \frac{d\lambda}{\vert {\bf c}(\lambda)\vert^2}\ r>0, 
\end{eqnarray}
where $d_g$ is the constant given by
\begin{eqnarray}\label{theconstant}
d_g = \frac{2^{(2Q-(n+2))} \Gamma(\frac{n}{2})}{\, \pi^{\frac{n}{2}+1} } 
\end{eqnarray}
and 
\begin{equation}\label{harishChandrac}
{\bf c}(\lambda) = 2^{(Q-2i\lambda)}\ \frac{\Gamma( \frac{n}{2}) \Gamma(2 i\lambda)}{\Gamma(\frac{n}{2}- \frac{Q}{2} + i\lambda) \Gamma(\frac{Q}{2} + i \lambda)}
\end{equation}
is called the Harish-Chandra $c$-function with respect to $\lambda\in {\Bbb C}$,
given in \cite[p.648]{ADY}. 

\medskip

${\rm Im}(\mathcal{H}^-) \subset C^{\infty, \mathrm{rad}}_0(X)$ is verified in Proposition \ref{compactsupportthm}, Section 8 by applying the so-called H\"ormander's trick given in \cite[Chap. I, Theorem 1.7.7]{Hormander} and \cite[sect. 4]{Flensted}. 

\medskip

The map $\mathcal{H}^-$ is the formal adjoint of $\mathcal{H}$ with respect to the $L_2$-inner products \eqref{ltwoinnerproduct}, \eqref{ltwoinnerproduct-2} on $C^{\infty,\mathrm{rad}}_0(X)$ and $\mathcal{PW}({\Bbb C})_\mathrm{even}$, respectively.
Therefore the inversion formula is stated in the following.
\begin{theorem}[\cite{ItohSatoh-2}]\label{inversion-1}
If $(X^n, g)$ is of $Q>0$ and of hypergeometric type, then any $f\in C^{\infty,\mathrm{rad}}_0(X)$ is recovered by
\begin{eqnarray}
f(r) = 2\, d_g \int_0^{\infty} {\hat f}(\lambda) \varphi_{\lambda}(r) \frac{d\lambda}{\vert {\bf c}(\lambda)\vert^2}.
\end{eqnarray}
\end{theorem} 
   
\begin{theorem}\label{convolutionrule-1}
Let $(X^n,g)$ be a simply connected harmonic Hadamard manifold of $Q>0$.
Then, the convolution rule for the spherical Fourier transform on $X$ holds as
\begin{eqnarray}
\widehat{ \left(f\ast f_1\right)}(\lambda) = \hat{f}(\lambda)\cdot \hat{f_1}(\lambda),\quad
f, f_1\in C^{\infty,\mathrm{rad}}_0(X).
\end{eqnarray} 
\end{theorem}
Here the convolution $f\ast f_1$ of $f, f_1\in C^{\infty,\mathrm{rad}}_0(X)$ is a function on $X$ defined by
\begin{eqnarray}
(f\ast f_1)(x):= \int_{y\in X} f(d(y,o)) f_1(d(x,y))\,dv_g(y),\ x\in X
\end{eqnarray}
Due to Z. Szab$\acute{\rm o}$ \cite{Sz} the convolution of radial functions turns out to be radial, since $X$ is simply connected and harmonic.
Theorem \ref{convolutionrule-1} is shown in Section 9.

\medskip

The Plancherel theorem, one of our main results, is obtained from the inversion formula, Theorem \ref{inversion-1} together with the convolution rule, Theorem \ref{convolutionrule-1}. 

\begin{theorem}\label{plancherel-1}
Let $(X^n,g)$ be a harmonic Hadamard manifold of $Q>0$.
Assume that $(X^n,g)$ be of hypergeometric type.
Then it holds for $f, f_1\in C^{\infty,\mathrm{rad}}_0(X)$
\begin{eqnarray}
\langle f, f_1\rangle_{(L_2, \omega_{n-1}\Theta dr)} 
= \langle {\hat f}, {\hat f}_1\rangle_{(L_2, d_g\vert{\bf c}(\lambda)\vert^{-2}d\lambda)}.
\end{eqnarray}
Here the inner products $\langle\cdot,\cdot\rangle_{(L_2, \omega_{n-1}\Theta dr)}$ on $C^{\infty,\mathrm{rad}}_0(X)$ and $\langle\cdot,\cdot\rangle_{(L_2,d_g \vert{\bf c}(\lambda)\vert^{-2}d\lambda)}$ on $\mathcal{PW}(\mathbb{C})_{\rm even}$ are defined respectively by
\begin{align}
\label{ltwoinnerproduct}
\langle f, f_1\rangle_{(L_2, \omega_{n-1}\Theta dr)}
:=& \omega_{n-1} \int_0^{\infty} f(r) \overline{f_1}(r) \Theta(r)\,dr,\ f, f_1\in C^{\infty,\mathrm{rad}}_0(X)\\
\label{ltwoinnerproduct-2}
\langle h, h_1\rangle_{(L_2, d_g \vert {\bf c}(\lambda)\vert^{-2}d\lambda)}
:=& 2\ d_g \int_0^{\infty} h(\lambda) \overline{h_1}(\lambda) \frac{d\lambda}{\vert{\bf c}(\lambda)\vert^2},\ h, h_1\in \mathcal{PW}({\Bbb C})_{\mathrm{even}}.
\end{align}
\end{theorem}

Recent results appeared in \cite{BKP} might be interesting.
The subject in \cite{BKP} is the non-spherical Fourier transform on a certain class of non-compact harmonic manifolds, wider than the class of harmonic manifolds of hypergeometric type. For definition of non-spherical Fourier transform refer to \eqref{nonspherical}, Section 4.
The proof of the inversion formula of non-spherical Fourier transform is there based on the spherical inversion formula, whose proof is given by applying the notion of hypergroup structure, a notion in Fourier analysis. In our setting, however, as indicated above and also in \cite{ItohSatoh-2} the proof is indebted to a simple geometric argument and a basic theorem in analysis.

\medskip

We outline the contents of this article as follows.
In Section 2 basic preliminaries for Hadamard manifolds are given and in Section 3 we treat notions and properties of Riemannian manifolds of non-positive sectional curvature which are adequate for formulating Riemannian geometry of Hadamard manifolds.
In Section 4 we introduce for a harmonic Hadamard manifold the spherical functions and the spherical Fourier transform.
Section 5 gives definition of the notion of hypergeometric type and Section 6 deals with the explicit form of the volume density and the mean curvature of geodesic spheres together with a remark on the curvature equalities associated to Ledger's formulas.
In Section 7 we prove Theorem \ref{volumeentropy-2} by using the Bishop comparison theorem and Theorem \ref{curvatureinequalities} by Ledger's formulas.
We verify in Section 8 Proposition 8.1 with respect to the dual map ${\mathcal{ H}}^-$. 
The convolution rule is verified in Section 9 and the Plancherel Theorem is verified in Section 10.
In Section 11 we give as an appendix the uniqueness theorem of the variable transformation for the spherical functions into the Gauss hypergeometric functions.

\medskip

The authors feel their thanks to Professor T. Arias-Marco for her encouraging them at their attending at DGA2019 Conference, University of Hradec Kr\'alov\'e.
Also the authors would like to thank the referee for valuable comments.

\section{Non-positive curvature and Riemannian manifolds}\label{twovolumedensity}

Let $(X^n,g)$ be a complete Riemannian manifold.
Assume $(X,g)$ is Hadamard. We recall basic geometric notions of a Hadamard manifold.

Let $\{r, \theta^i,\, i= 1,\dots,n-1\}$ be geodesic polar coordinates around $y\in X$.
The Laplace-Beltrami operator $\Delta$ at a point $x= \exp_y r u$, $u\in S_yX$ is represented by
\begin{eqnarray}\label{radiallaplace}
\Delta = - \left( \frac{\partial^2}{\partial r^2} + \sigma_y(\exp_y ru) \frac{\partial}{\partial r}\right) + {\tilde{\Delta}}
\end{eqnarray}
where $\sigma_y(\exp_y ru)$ is the mean curvature of the geodesic sphere $S(y;r)$ at $x$ and ${\tilde{\Delta}}$ is the Laplace-Beltrami operator of $S(y;r)$ at $x$ (\cite[(1.2)]{Sz}).

Let $\gamma$ be a geodesic given by $\gamma(t):= \exp_y tu$, where $u\in S_yX$ 
($S_yX$ denotes the set of all unit tangent vectors at $y$) and let $\{E_i= E_i(t), \, i=1,\cdots,n\}$ be a parallel orthonormal frame field along $\gamma$ with $E_1(t) = \gamma'(t)$.
Let $Y_i(t)$, $i=2,\cdots,n$ be a perpendicular Jacobi vector field along $\gamma$ for $t\geq 0$ satisfying  $Y_i(0) = 0$ and $Y'_i(0) = E_i(0)$, $i=2,\cdots,n$.
Then, the square root determinant
\begin{eqnarray}\label{squarerootdet}
\Theta_y(\exp_y tu) := 
\sqrt{\det\left(g(Y_i(t),Y_j(t)) \right)}_{2\leq i,j\leq n}
\end{eqnarray}
yields the volume density of $S(y;t)$ at $x= \exp_y tu$.
The volume entropy is defined by $\displaystyle{Q= \lim_{r\rightarrow\infty} \frac{1}{r} \log {\rm Vol}\, B(y;r)}$, where ${\rm Vol}\, B(y;r)$ is the volume of the closed ball $B(y;r)$; $\displaystyle {\rm Vol}\, B(y;r) = \int_0^r dt \int_{u\in S_yX} \Theta_y(\exp_y tu)\,du$. 

\begin{lemma}\label{formula}
\begin{eqnarray}
\sigma_y(\exp_y tu) = \frac{\frac{\partial}{\partial t}\Theta_y(\exp_y tu)}{\Theta_y(\exp_y tu)},\, t > 0.
\end{eqnarray}
\end{lemma} 

In fact, let $t>0$ and $\gamma(t)^{\perp}=\{ v\in T_{\gamma(t)}X\, ;\, v \perp \gamma'(t)\}$ and $A(t) : \gamma(t)^{\perp}\rightarrow \gamma(t)^{\perp}$ be an endomorphism of $\gamma^{\perp}(t)$, defined by $A(t) E_i(t) = Y_i(t)$, $i=2,\cdots,n$.
Then, $A(t)$ satisfies 
\begin{eqnarray}\label{jacobitensor}
A''(t) + R_{\gamma'(t)}\circ A(t) = 0,\ A(0) = 0, A'(0) = {\rm Id}_{u^\perp}.
\end{eqnarray} 
Here $R_{\gamma'(t)}$ is the Jacobi operator associated with the Riemannian curvature tensor $R$ of $(X^n,g)$;
$R_{\gamma'(t)} : \gamma(t)^{\perp}\rightarrow \gamma(t)^{\perp}$,
$v \mapsto R(v,\gamma'(t))\gamma'(t)$.
Using $A(t)$ we define an endomorphism ${\mathcal S}(t) := A'(t)\circ A^{-1}(t)$ of $\gamma(t)^{\perp}$.
From the non-positivity of sectional curvature $A(t)$ is invertible for $t>0$ so that ${\mathcal S}(t)$ turns out to be self-adjoint from $A(0)= 0$, $A'(0)= {\rm Id}_{u^{\perp}}$.
${\mathcal S}(t)$ and its trace ${\rm Tr}\  {\mathcal S}(t) =: \sigma_y(\exp_y tu)$ are the shape operator and the mean curvature of $S(y;t)$ at $\gamma(t)$, respectively.
The family of shape operators $\{ {\mathcal {S}}(t)\,;\, 0<t<\infty\}$ satisfies the Riccati equation
\begin{eqnarray}\label{riccati}
{\mathcal{S}}'(t) + {\mathcal{S}}^2(t)+R_{\gamma'(t)} = 0
\end{eqnarray}
along $\gamma(t)$.
Now one obtains from \eqref{squarerootdet}
\begin{eqnarray}\label{det}
\Theta_y(\exp_y tu) = \det A(t),
\end{eqnarray}
so that $\displaystyle \frac{\partial}{\partial t}\Theta_y(\exp_y tu) = {\rm Tr}\, {\mathcal{S}}(t)\, \Theta_y(\exp_y tu)$ showing the lemma. 

\begin{lemma}\label{derivativesofq}
\begin{eqnarray}
\lim_{t\rightarrow 0} \frac{\Theta_y(\exp_y tu)}{t^{n-1}} = 1,\ 
\lim_{t\rightarrow 0} \frac{\partial}{\partial t}\, \left(\frac{\Theta_y(\exp_y tu)}{t^{n-1}}\right) = 0
\end{eqnarray}
and
\begin{eqnarray}\label{ledgerformulaelse}
\left.\frac{\partial^2}{\partial t^2}\left( \frac{\Theta_y(\exp_y tu)}{t^{n-1}}\right)\right\vert_{t=0} = - \frac{1}{3} {\rm{Ric}}_u,
\end{eqnarray}
where ${\rm{Ric}}_u$ is the Ricci curvature of $u\in S_yX$ given by ${\rm{Ric}}_u ={\rm Tr}\, R_u$.
\end{lemma} 

For the formula \eqref{ledgerformulaelse} refer to \cite[p.1]{Sz}.

\begin{proof}
From \eqref{jacobitensor} $A(t)$ and $\Theta_y(\exp_y tu)$ have Taylor expansions of $t$ as 
 
\begin{eqnarray}\label{expansionofa}
A(t) = t\left({\rm Id}_{u^{\perp}} - \frac{1}{3!} R_u t^2 + O(t^3)\right) 
\end{eqnarray}
and
\begin{eqnarray}\label{expand}
\Theta_y(\exp_y tu)=\det A(t) = t^{n-1}\left(1- \frac{1}{3!} {\rm Ric}_u t^2 + o(t^2) \right),
\end{eqnarray}
respectively, from which one gets the formulas of the lemma. 
\end{proof}
 
Here and in what follows, an endomorphism ${\mathcal T}(t)$ of $\gamma(t)^{\perp}$ is often regarded as its corresponding matrix by means of the parallel orthonormal frame field $\{E_i(t)\}$, unless other wise confused.
For \eqref{ledgerformulaelse} see \cite[p.1]{Sz}.

\begin{remark}\label{expansionremark}
From \eqref{expansionofa} ${\mathcal S}(t)= A'(t)\circ A^{-1}(t)$ is also expanded for a sufficiently small $t>0$ along $\gamma(t)$ as $\displaystyle{ {\mathcal S}(t) = \frac{1}{t} {\rm Id}_{u^{\perp}} + O(t)}$ and hence one has
\begin{eqnarray}
\sigma_y(\exp_y tu) = \frac{n-1}{t} + O(t).
\end{eqnarray}
\end{remark}

On the other hand, by setting $\tau_y(\exp_y tu) = \sigma_y(\exp_y tu) - (n-1)/t$ one has from Lemma \ref{formula}
\begin{eqnarray} 
\tau_y(\exp_y tu) = \frac{\partial}{\partial t} \left(\log \frac{\Theta_y(\exp_y tu)}{t^{n-1}}\right)
\end{eqnarray}
and from \eqref{expand} $\tau_{y}(\exp_y tu)\rightarrow 0$ as $t\rightarrow 0$. 
Therefore
\begin{align}
\lim_{t\rightarrow 0}\, \frac{\partial}{\partial t} \tau_y(\exp_y tu) 
=& \nonumber \lim_{t\rightarrow 0}\, \frac{\partial^2}{\partial t^2}\left(\log \frac{\Theta_y(\exp_y tu)}{t^{n-1}}\right)
 = \lim_{t\rightarrow 0}\, \frac{\partial^2}{\partial t^2}\left( \frac{\Theta_y(\exp_y tu)}{t^{n-1}}\right)\\
 =& \left.\frac{\partial^2}{\partial t^2}\left( \frac{\Theta_y(\exp_y tu)}{t^{n-1}}\right)\right\vert_{t=0}\label{firstderivative},
\end{align}
by using the formula;
\begin{eqnarray*}
\frac{\partial^2}{\partial t^2} \log F(t) =\frac{\frac{\partial^2 F}{\partial\, t^2} }{ F(t)} - \frac{\left(\frac{\partial F}{\partial\, t} \right)^2}{F(t)^2},\ 
F(t) = \frac{\Theta(t)}{t^{n-1}}.
\end{eqnarray*}
Notice the mean curvature $\sigma_y$ satisfies for any $u$
\begin{eqnarray}\label{meancurvature}
\sigma_y(\exp_y tu) - \frac{n-1}{t} \rightarrow 0\ \mbox{as}\ t\rightarrow 0.
\end{eqnarray}

\section{Busemann Function and Horospheres of a Hadamard manifold}

Let $(X^n,g)$ be a Hadamard manifold of $Q>0$ and of dimension $n\, (\geq 3)$.
In this section the existence of spherical functions on $(X^n,g)$ is shown, when $(X^n,g)$ is harmonic.
For this purpose we will exhibit the existence of certain eigenfunctions of $\Delta$, called $\lambda$-Poisson kernels, by using Busemann function on $X$ and averaging it over geodesic spheres. So, we begin with introducing Busemann function and horospheres, the level hypersurfaces of Busemann function together with the ideal boundary in order to define the spherical functions on a harmonic Hadamard manifold.

The non-positivity of sectional curvature of $(X^n,g)$ implies that any non-trivial geodesic $\gamma(t)$ on $X$ is conjugate free so that $(X^n,g)$ is diffeomorphic to ${\Bbb R}^n$, known as Cartan-Hadamard theorem.
Notice that a harmonic Hadamard manifold $(X^n,g)$ of $Q=0$ must be a Euclidean space. See for this \cite{RS}.

Any Hadamard manifold $(X^n,g)$ admits in a natural sense the ideal boundary $\partial X$, a boundary at infinity by taking quotient of the space of all geodesic rays of unit speed with respect to the asymptotical equivalence.
The ideal boundary $\partial X$ can be identified with the unit tangent sphere $S_yX$ for any fixed $y\in X$.

Let $o\in X$ be a fixed, reference point and $\gamma$ be a geodesic on $X$ satisfying $\gamma(0) = o$, $\gamma'(0)=u\in S_oX$. Let $b_{\gamma}$ be the Busemann function on $X$, associated with $\gamma$ defined by
\begin{eqnarray*}
b_{\gamma}(x) := \lim_{t\rightarrow\infty} \{d(x, \gamma(t))-t\},\ x\in X.
\end{eqnarray*}
See \cite{Heintzeimhpf, BGS, Eberlein,Sakai} and \cite{ItohSatohKyushu} for basic properties of Busemann function and related references.
The Busemann function $b_{\gamma}$ is of $C^2$ and satisfies $\vert b_{\gamma}(x)\vert \leq d(o,x)$, $\forall x\in X$ and further $\vert \nabla b_{\gamma}\vert \equiv 1$ everywhere.
From the non-positivity of curvature, $b_{\gamma}$ is convex and then the Hessian $\nabla d b_{\gamma}$ is positive semi-definite. 

Let $\theta\in\partial X$ be an ideal boundary point represented by the geodesic ray $\gamma$, so $\theta=[\gamma]$.
Let $y$ be a point of $X$, $y\not= o$ and $\gamma_1$ be a geodesic ray of
$\gamma_1(0)= y$, asymptotically equivalent to $\gamma$ so that $\theta=[\gamma_1]$.
The Busemann functions $b_{\gamma}$ and $b_{\gamma_1}$ associated with $\gamma$ and $\gamma_1$, respectively satisfy 
\begin{eqnarray}\label{coclosedcondition}
b_{\gamma}(x) = b_{\gamma_1}(x) + b_{\gamma}(y),\, x\in X.
\end{eqnarray}
Thus, asymptotically equivalent geodesics and then the ideal boundary point represented by them defines the same gradient field $\nabla b_{\gamma}$ and the same Hessian $\nabla d b_{\gamma}$. Fix $o\in X$.
Then for any $\theta\in\partial X$ we define the Busemann function $b_{\theta}$ associated with $\theta$ by $b_{\theta}(x)=b_{\gamma}(x)$, $x\in X$, using the geodesic ray $\gamma$ of $\gamma(0)=o, [\gamma]=\theta$.

A level hypersurface of the Busemann function $b_{\theta}$ associated to $\theta$ of level $t\in {\Bbb R}$ is called a horosphere centered at $\theta$, which we denote by $H_{(t,\theta)}$; $H_{(t,\theta)} = \{ y\in X;\, b_{\theta}(y)=t\}$.
Each horosphere is diffeomorphic to ${\Bbb R}^{n-1}$ and admits the induced metric from $(X^n,g)$.
Here $\nabla b_{\theta}$ and $\nabla d b_{\gamma}$ are viewed respectively the unit normal and the second fundamental form of the horosphere $H_{(t,\theta)}$ at $x\in H_{(t,\theta)}$.
The Laplacian $\Delta b_{\theta}(x) = - {\rm {Tr}}\, \nabla d b_{\theta}$ is then viewed as the mean curvature of $H_{(t,\theta)}$.
 
For a fixed $\theta\in\partial X$ the family of horospheres $\{H_{(t,\theta)}\,; t\in {\Bbb R}\}$ centered at $\theta$ foliates $X$ so that one has a horospherical fibration $\pi \, :\, X \rightarrow {\Bbb R}$ of $X$ with horospheres as fibers and with the projection given by $\pi(x)= b_{\theta}(x)$. Note that this fibration is a Riemannian submersion.

\begin{proposition}
\label{q} \rm Let $(X^n,g)$ be a Hadamard manifold of $Q>0$.
Assume $(X,g)$ is harmonic.
Then $(X^n,g)$ is asymptotically harmonic, that is, every horosphere of $(X^n,g)$ has constant mean curvature and the value of this constant is $-Q$, common for all horospheres.
Then each Busemann function $b_{\theta}$, $\theta\in \partial X$ satisfies $\Delta b_{\theta} = - Q$ and is analytic.
 \end{proposition} For the terminology of asymptotical harmonicity refer to \cite{Led}. Each horosphere is a geometrical limit of geodesic spheres so that the harmonicity of a Hadamard manifold implies asymptotical harmonicity, as indicated in \cite{K2}. So one gets the proposition. 

\section{Spherical functions and spherical Fourier transform}

Let $(X^n,g)$, $n \geq 3$ be a harmonic, Hadamard manifold of $Q>0$.

\begin{definition}
A radial eigenfunction $\psi_{\lambda}= \psi_{\lambda}(r)$ of the Laplace-Beltrami operator $\Delta $ on $X$, normalized by $\psi_{\lambda}(0)=1$ is called a {\it spherical function} of $(X,g)$.
See \cite[Intro]{H}.
\end{definition}

\begin{definition}\label{lambdapoisson}
A function $P(x,\theta)$ on $X$ associated with $\theta=[\gamma]$ defined by
\begin{eqnarray}
P(x,\theta) := \exp\{-Q\,b_{\theta}(x)\},\ x\in X
\end{eqnarray}
is called {\it Poisson kernel} and the function $P_{\lambda}(x,\theta)$ on $X$ defined by 
\begin{eqnarray}
P_{\lambda}(x,\theta) = \left\{P(x,\theta) \right\}^{(1/2 - i \lambda/Q)} = \exp\left\{\left(-\frac{1}{2}Q + i\lambda\right) b_{\theta}(x) \right\},\ x\in X,\, \lambda\in {\Bbb C},
\end{eqnarray}
is called $\lambda$-{\it Poisson kernel}.
\end{definition}

For the Poisson kernel on a Damek-Ricci space refer to \cite{ItohSatohTokyo, Damek}.
 
\medskip

Since $(X^n,g)$ is asymptotically harmonic, from Proposition \ref{q}, one has $\Delta b_{\theta} = - Q$, $\theta\in\partial X$ and then the following.
\begin{proposition}
For any fixed $\theta$, $P_{\lambda}(x,\theta)$ is an eigenfunction of $\Delta$ with eigenvalue $\displaystyle{\frac{Q^2}{4}+\lambda^2}$;
\begin{eqnarray}
\Delta P_{\lambda}(x,\theta) = \left(\frac{Q^2}{4}+\lambda^2\right) P_{\lambda}(x,\theta).
 \end{eqnarray}
\end{proposition}
This is from a straightforward computation. 

\medskip

Averaging the $\lambda$-Poisson kernel over a geodesic sphere $S(o;r)$ yields a radial eigenfunction $\varphi_{\lambda}(r)$;
\begin{eqnarray}\label{sphericalfunction}
\varphi_{\lambda}(r(x)) := {\mathcal{M}}_o(P_{\lambda}(x,\theta)),\, r(x) = d(o,x).
\end{eqnarray}
One sees $\varphi_{\lambda}(0) = 1$ and $\varphi_{\lambda}'(0) = 0$.
Moreover, $\varphi_{\lambda}(r)$ is a solution of 
\begin{eqnarray}\label{sphricalequa}
-\left(\frac{d^2}{dr^2} + \sigma(r) \frac{d}{dr}\right) \varphi_{\lambda}
= \left(\frac{Q^2}{4}+\lambda^2\right) \varphi_{\lambda}.
\end{eqnarray}

\begin{proposition}\label{sphericalfunction-2}
The radial function $\varphi_{\lambda}(r)$, $\lambda\in {\Bbb C}$ defined by \eqref{sphericalfunction} gives the spherical function on $(X^n,g)$ of eigenvalue $Q^2/4+ \lambda^2$.
\end{proposition}

\begin{note}\label{note-1}
The uniqueness of spherical function is guaranteed by \eqref{sphricalequa}.
We extend $\varphi_{\lambda}$ as an even function on ${\Bbb R}$.
From \eqref{sphricalequa} $\varphi_{\lambda} = \varphi_{\mu}$ if and only if $\lambda = \pm \mu$.
So, $\varphi_{\lambda}(r) = \varphi_{-\lambda}(r)$ and $\overline{\varphi_{\lambda}(r)} = \varphi_{\overline{\lambda}}(r)$,\, $r \in {\Bbb R}$.
$\varphi_{\lambda}(r)$, $r\in {\Bbb R}$, is real valued for $\lambda\in {\Bbb R}$.
\end{note}

\begin{lemma}\label{estimationofspherical}
\begin{enumerate}
\item $\varphi_{\lambda}(r)\equiv 1$ for $\lambda=\pm i Q/2$.
\item If $\vert{\rm Im} \lambda\vert \leq Q/2$, then $\vert \varphi_{\lambda}(r)\vert \leq 1$ for $\forall r\in{\Bbb R}$.
\item Moreover, $\vert \varphi_{\lambda}(r)\vert \leq \exp\left(Qr/2\right)$ for $\forall \lambda\in {\Bbb R}$.
\end{enumerate}
\end{lemma}

\begin{proof}
(i) is obvious.
See \cite{ItohSatoh-2} for the proof of (ii).
(iii) is shown in the following way.
From definition we have
\begin{eqnarray}
\vert \varphi_{\lambda}(r(x))\vert \leq \frac{1}{{\rm{Vol}}(S(o;r))} \int_{x\in S(o;r)} \left\vert \exp\left\{\left(-\frac{Q}{2}+i\lambda\right) b_{\gamma}(x)\right\}\right\vert dv_{S(o;r)}
\end{eqnarray}
Since $\displaystyle{- b_{\gamma}(x) \leq r(x)}$ for $\lambda\in {\Bbb R}$, as indicated in Section 3, one has
\begin{eqnarray}
\left\vert \exp\left\{\left(-\frac{Q}{2}+i\lambda\right) b_{\gamma}(x)\right\}\right\vert
 = \exp\left(-\frac{Q}{2}b_{\gamma}(x)\right)\leq \exp\left(\frac{Q}{2}r(x)\right)
\end{eqnarray}
from which (iii) is derived.
\end{proof}

Let $\{\varphi_{\lambda}\, ;\, \lambda\in{\Bbb C}\}$ be the family of spherical functions parametrized by $\lambda$, defined at \eqref{sphericalfunction} on a harmonic Hadamard manifold $(X^n, g)$ of $Q > 0$.

Let $f = f(r(x))$ be a smooth radial function on $X$ of compact support. We have then from Definition \ref{definition-1} the spherical Fourier transform ${\hat f}$ of $f$ by
\begin{equation}\label{sphericaltransform}
{\hat f}(\lambda): = \int_X f(r(x)) \varphi_{\lambda}(r(x))\,dv_g(x)\\ 
= \omega_{n-1} \int_0^{\infty} f(r) \varphi_{\lambda}(r)\, \Theta(r)\, dr,\ \lambda\in {\Bbb C},
\end{equation}
where
\begin{equation}\label{volumeunitsphere}
\omega_{n-1} := {\rm Vol}(S^{n-1}(1))= \frac{2 \pi^{n/2}}{\Gamma(\frac{n}{2})}.
\end{equation}
Note ${\rm Vol}(S(o;r)) = \omega_{n-1} \Theta(r)$.
Refer to \cite{H} for the definition over symmetric spaces. 

We have then a linear map
\begin{equation}
{\mathcal H} : C^{\infty, \mathrm{rad}}_0(X) \rightarrow \mathcal{PW}({\Bbb C})_{\mathrm{even}} ;\quad
f= f(r) \mapsto {\hat f}(\lambda)
\end{equation}
from the space $C^{\infty, \mathrm{rad}}_0(X)$ into the space $\mathcal{PW}({\Bbb C})_{\mathrm{even}}$ of even holomorphic functions $h = h(\lambda)$ of $\lambda\in {\Bbb C}$ of exponential type. See \cite{ADY,R}.
Refer also to \cite{ItohSatoh-2}.
\begin{definition}\label{paleywiener}
The Paley-Wiener space $\mathcal{PW}({\Bbb C})_{\mathrm{even}}$ is defined by 
\begin{eqnarray}
\mathcal{PW}({\Bbb C})_{\mathrm{even}}
= \bigcup_{R>0} \mathcal{PW}({\Bbb C})^R_\mathrm{even}
\end{eqnarray}
where $\mathcal{PW}({\Bbb C})^R_\mathrm{even}$ is the space of even, holomorphic functions $h=h(\lambda)$ on ${\Bbb C}$ satisfying the following; for any $N\in {\Bbb N}$ there exists a constant $c_N > 0$ such that 
\begin{eqnarray}\label{exponentialtype}
\vert h(\lambda)\vert \leq c_N (1+\vert\lambda\vert)^{-N} \exp(R\,\vert {\rm Im}\, \lambda\vert),\ \forall \lambda\in {\Bbb C}.
\end{eqnarray} 
\end{definition}

Let $f \in C^{\infty, \mathrm{rad}}_0(X)$.  
By using the horospherical foliation $\left\{H_{(t,\theta)}\right\}$ of $X$ parametrized in $t\in{\Bbb R}$ with a fixed $\theta\in\partial X$, one defines
\begin{align}
g(t) :=& \int_{x\in H_{(t,\theta)}}
 f(r(x)) \exp\left(-\frac{Q}{2}t\right)\,dv_{H_{(t,\theta)}}(x)\\
 =& \label{gfunction} e^{-\frac{Q}{2}t} \int_{x\in H_{(t,\theta)}} f(r(x))\,dv_{H_{(t,\theta)}}(x).
\end{align}
Note $g=g(t)$ is smooth in $t$ and of a compact support. 
\begin{definition}
The linear map defined by 
\begin{eqnarray}
{\mathcal A}\, : C^{\infty, \mathrm{rad}}_0(X) \, &\rightarrow& C^{\infty}_0({\Bbb R})\, ;\  f \mapsto g
\end{eqnarray}
is called Abel transform.
Here the function $g(t)$ is defined as above.
\end{definition} 

\begin{proposition} 
The spherical Fourier transform ${\hat f}$ of $f$ has the form of 
\begin{eqnarray}{\hat f}(\lambda) = \int_{-\infty}^{\infty} {\mathcal{A}}(f)(t) \exp{(i\lambda t)} dt.
\end{eqnarray}

Moreover, the image of ${\mathcal H}$ is contained in $\mathcal{PW}({\Bbb C})_{\mathrm{even}}$.
\end{proposition} 

\begin{proof}
Since 
\begin{eqnarray}
\int_{S(o;r)} \varphi_{\lambda}(r(x))\,dv_{S(o;r)} = {\rm Vol}(S(o;r)) \varphi_{\lambda}(r) = \int_{S(o;r)} P_{\lambda}(x,\theta)\,dv_{S(o;r)},
\end{eqnarray}
one has
\begin{equation}
{\hat f}(\lambda) = \omega_{n-1} \int_0^{\infty} f(r) \varphi_{\lambda}(r)\, \Theta(r)\,dr \\
=
\int_{x\in X} f(r(x)) P_{\lambda}(x,\theta)\,dv_g(x).
\end{equation}
We make use of the horospherical foliation $\{H_{(t,\theta)}\, ;\, t\in {\Bbb R}\}$ of $X$ as
\begin{align}\label{lambdapoissonkernelexpression}
{\hat f}(\lambda)
=&\int_{-\infty}^{\infty} dt \int_{x\in \mathcal{H}_{(t,\theta)}} f(r(x)) \exp\left\{\left(-\frac{Q}{2}+ i\lambda\right) b_{\theta}(x)\right\} dv_{H_{(t,\theta)}}(x)\nonumber\\
=& \int_{-\infty}^{\infty} dt \exp\left\{\left(-\frac{Q}{2}+ i\lambda\right) t\right\} \int_{x\in \mathcal{H}_{(t,\theta)}} f(r(x))\,dv_{H_{(t,\theta)}}(x)\nonumber\\
=& \int_{-\infty}^{\infty} \exp ( i\lambda t)\cdot g(t)\,dt
\end{align}
for a fixed $\theta$.
Then this indicates that ${\hat f}$ is the classical Fourier transform of the compactly supported, smooth function $g(t)$.
Thus ${\hat f}(\lambda)$ belongs to the Paley-Wiener space, as one of basic properties of the classical Fourier transform on ${\Bbb R}$.
In fact, this follows from
\begin{eqnarray}
\vert {\hat f}(\lambda)\vert \leq \int_{-R}^{R} \vert g(t)\vert \vert \exp (i\lambda t)\vert dt\leq \int_{-R}^R \vert g(t)\vert dt\, \cdot e^{R\vert {\rm Im} \lambda\vert}.
\end{eqnarray} 
For the detailed argument refer to \cite{ItohSatoh-2}.
\end{proof} 

\begin{theorem}[{\cite[Theorem 3.12]{PS}}, {\cite[(2.9)]{ADY}}]\label{bijective}
$\mathcal{H}$ is bijective.
\end{theorem}
Remark that in \cite{ADY} the map $\mathcal{H}$ for a Damek-Ricci space gives an isomorphism, and in \cite{PS} $\mathcal{H}$ is shown to be isomorphic in the sense of distributions. 

\medskip

Let $f$ be a smooth function on $X$ of compact support, not necessarily radial. Then, it is a routine matter to define Fourier transform of $f$ by the aid of the $\lambda$-Poisson kernel as a function on ${\Bbb C}\times \partial X$;
\begin{eqnarray}\label{nonspherical}
{\hat f}(\lambda,\theta) = \int_X f(x)\ P_{\lambda}(x,\theta)\,dv_g(x).
\end{eqnarray}
For this transform, called the Helgason-Fourier transform refer to \cite{H, ACDB}, when $(X,g)$ is a symmetric space or a Damek-Ricci space.
Refer also to \cite{BKP} by using the results of which we would like to presume to state the inversion formula and the Plancherel theorem for the non-spherical Fourier transform \eqref{nonspherical} on a harmonic Hadamard manifold as follows, when it is of hypergeometric type (a notion of special type, explained in the next section)
\begin{eqnarray}
f(x) = 2 d_g \int_{\theta\in\partial X}\, \int_0^{\infty} {\hat f}(\lambda,\theta) P_{\lambda}(x,\theta) d\lambda_o(\theta) \frac{d\lambda}{\vert{\bf c}(\lambda)\vert^2}
\end{eqnarray}
and
\begin{eqnarray}
\int_X f(x){\overline f_1}(x) dv_g(x) = 2 d_g \int_{\theta\in\partial X}\, \int_0^{\infty} {\hat f}(\lambda,\theta)\overline{\hat{f_1}}(\lambda,\theta)d\lambda_o(\theta) \frac{d\lambda}{\vert{\bf c}(\lambda)\vert^2} 
\end{eqnarray}
for $f, f_1\in C^{\infty}_o(X)$, respectively.
Here $d\lambda_o$ denotes the push-forward measure on $\partial X$ from the standard measure on $S^1_oX$. Notice that when $f$ is radial, then ${\hat f}(\lambda,\theta)$ coincides with the spherical Fourier transform of $f$.
 
\section{Hypergeometric type}

\begin{definition}[\cite{ItohSatoh-2}]\label{definition}
A harmonic Hadamard manifold $(X^n,g)$ of $n\geq 3$ and of volume entropy $Q > 0$ is said to be {\it of hypergeometric type}, when, by the transformation $z = - \sinh^2 r/2$ of the radial variable $r$, the solution $\varphi_{\lambda}$ of the equation 
\begin{eqnarray}\label{eigenequation}
\Delta \varphi = - \left(\frac{d^2}{dr^2}+ \sigma(r)\frac{d}{dr}\right) \varphi = \left(\frac{Q^2}{4}+\lambda^2\right) \varphi,\ \varphi(0)=1
\end{eqnarray}
is converted into a Gauss hypergeometric function $F(a,b,c;z)$, i.e., a solution $f=f(z)$ of the Gauss hypergeometric equation
\begin{equation}\label{hypergeometricequation}
z (1-z) \frac{d^2 f}{d z^2} + \left(c -(a+b+1)z\right)\frac{df}{dz} - a b\, f = 0,
\end{equation}
where $a, b, c \in \mathbb{C}$ are constants, and $c \not= 0, -1, -2, \cdots$. 
 \end{definition}

We have then the following.
\begin{theorem}[\cite{ItohSatoh-2}]\label{hypergeometrictype-1}
Let $(X,g)$ be a harmonic Hadamard manifold of $Q>0$.
Assume $(X,g)$ is of hypergeometric type. Let $\lambda\in {\Bbb C}$ be arbitrary. 
Then the spherical function $\varphi_{\lambda}(r)$ on $X$ with parameter $\lambda$ is described as
\begin{eqnarray}\label{gausshypergeometricexpression}
\varphi_{\lambda}(r) = F\left(\frac{Q}{2} - i \lambda,\frac{Q}{2} + i \lambda, \frac{n}{2}; z\right),\, z= - \sinh^2 \frac{r}{2}.
\end{eqnarray}
Moreover the mean curvature $\sigma(r)$ of geodesic sphere $S(o,r)$ is represented by
\begin{eqnarray}\label{meancurvatureexpression}
\sigma(r) = \frac{n-1}{2}\ \coth \frac{r}{2} + \left(Q-\frac{n-1}{2}\right) \tanh \frac{r}{2}.
\end{eqnarray}
Here $F(a,b,c;z)$ is the Gauss hypergeometric function with parameters $a, b, c$. 
\end{theorem}

The proof is achieved in \cite{ItohSatoh-2}.
The equation \eqref{hypergeometricequation} to which the function $f(z)$, defined by $f(z(r)) = \varphi_{\lambda}(r)$ is a solution admits the parameters $a, b$ and $c$ as  $\displaystyle a=\frac{Q}{2} \pm i \lambda,\, b= \frac{Q}{2} \mp i \lambda$ and $\displaystyle c= \frac{n}{2}$.

\begin{note}
Suppose that there exists at least one $\lambda\in{\Bbb R}$ such that the equation \eqref{eigenequation} for the $\varphi_{\lambda}$ of $Q^2/4+\lambda^2 \not=0$ is converted into the equation \eqref{hypergeometricequation} by $z=-\sinh^2 r/2$.
Then the equation \eqref{eigenequation} associated with any $\lambda\in{\Bbb C}$ is converted by the same transformation into the equation \eqref{hypergeometricequation} of certain parameters $a,b,c$.
Here the parameters $a, b$ depend on $\lambda\in {\Bbb C}$.
\end{note}

\begin{remark}\label{jacobifunction}
If $(X,g)$ is of hypergeometric type, then the equation \eqref{sphericaltransform} becomes from Theorem \ref{hypergeometrictype-1}
\begin{align}
\Delta^\mathrm{rad} \varphi
:= & - \left\{\frac{d^2}{dr^2} + \left(\frac{n-1}{2}\ \coth \frac{r}{2} + \left(Q-\frac{n-1}{2}\right) \tanh \frac{r}{2}\right) \frac{d}{dr} \right\}\varphi\nonumber\\
 = & \left(\frac{Q^2}{4}+ \lambda^2\right) \varphi,
\end{align}
which turns out to be
\begin{eqnarray}\label{eigenjacobiequation}-\left( \frac{d^2}{dt^2} + \left\{(2\alpha+1)\coth t + (2\beta+1)\tanh t\right\} \frac{d}{dt}\right)\phi = \left(Q^2 + \mu^2\right)\phi,
\end{eqnarray}
when one substitutes $t= r/2$, where
\begin{eqnarray}\label{alphabeta}
\alpha= \frac{n}{2}-1,\ \beta= Q- \frac{n}{2},\ \mu= 2\lambda.
\end{eqnarray}
The solution $\phi(t)= \phi_{\mu}^{(\alpha,\beta)}(t)$, $\phi(0)= 1$ to the equation \eqref{eigenjacobiequation} is called the Jacobi function of order $(\alpha,\beta)$. Each Jacobi function $\phi_{\mu}^{(\alpha,\beta)}(t)$ of \eqref{eigenjacobiequation} for arbitrary $(\alpha,\beta)$ is represented by the hypergeometric function
\begin{eqnarray}
\phi_{\mu}^{(\alpha,\beta)}(t) = F\left(\frac{T-i\mu}{2}, \frac{T+i\mu}{2}, \alpha+1; - \sinh^2 t\right),\ T:=\alpha+\beta+1. 
\end{eqnarray}
Another solution of \eqref{eigenjacobiequation} is given by the function
\begin{eqnarray}\label{jacobiseconkind}
\Phi_{\mu}^{(\alpha,\beta)}(t) = \left(2\sinh t)\right)^{i\mu-T}\, F\left(\frac{-\alpha+\beta+1-i\mu}{2}, \frac{T-i\mu}{2}, 1-i\mu; -\sinh^{-2}t\right)
\end{eqnarray}
for $\mu\not\in - i{\Bbb N}$.
Remark that $\Phi_{\mu}^{(\alpha,\beta)}(t)$ is characterized by that $\displaystyle{\Phi_{\mu}^{(\alpha,\beta)}(t) = e^{(i\mu-T)t}(1+0(1))}$, $t\rightarrow\infty$.
The functions $\phi_{\mu}^{(\alpha,\beta)}(t)$, $\Phi_{\mu}^{(\alpha,\beta)}(t)$, called Jacobi functions of first and second kind satisfy the following
\begin{align}
\label{linearcombi-2}
\frac{\sqrt{\pi}}{\Gamma(\alpha+1)} \phi_{\mu}^{(\alpha,\beta)}(t)
=&\frac{1}{2}c_{\alpha,\beta}(\mu) \Phi_{\mu}^{(\alpha,\beta)}(t) + \frac{1}{2}c_{\alpha,\beta}(-\mu) \Phi_{-\mu}^{(\alpha,\beta)}(t),\\ 
\label{c}
c_{\alpha,\beta}(\mu)
=& \frac{2^T\ \Gamma(\frac{i\mu}{2}) \Gamma(\frac{1}{2}(1+i\mu))}{\Gamma(\frac{1}{2}(T+i\mu)) \Gamma(\frac{1}{2}(\alpha-\beta+1+i\mu))}
\end{align}
for $\mu\not\in {\Bbb Z}$.
So the spherical function $\varphi_{\lambda}(r)$ satisfies
\begin{eqnarray}\label{connectionformula}
\varphi_{\lambda}(r) = \phi_{\mu}^{(\alpha,\beta)}(t) = \frac{\Gamma(\frac{n}{2})}{2\sqrt{\pi}}\Big\{c_{\alpha,\beta}(\mu) \Phi_{\mu}^{(\alpha,\beta)}(t)+ c_{\alpha,\beta}(-\mu) \Phi_{-\mu}^{(\alpha,\beta)}(t)
\Big\},
\end{eqnarray}
where $t= \frac{r}{2}$, $\mu=2\lambda$ and $\alpha= \frac{n}{2}-1$, $\beta= Q-\frac{n}{2}$.
This relation which stems exactly from the connection formula for hypergeometric functions is significantly important for the argument of the spherical Fourier transform.
 \end{remark}
  
\section{Volume density of geodesic spheres}

Let $(X^n,g)$ be a harmonic Hadamard manifold of volume entropy $Q >0$.
Assume $(X^n,g)$ is of hypergeometric type. 

The aim of this section is to prove \eqref{volumedensity-1} of Theorem \ref{volumedensmeancurv}.
 
\begin{proof}[Proof of \eqref{volumedensity-1} of Theorem \ref{volumedensmeancurv}]

In Theorem \ref{geomcharacterization} the constants $c_1, c_2$ appeared in \eqref{volumedensity} are fixed as $c_1=(n-1)/2$, $c_2= (2Q-(n-1))/2$ from the formula \eqref{meancurvatureexpression}.
The constant $k > 0$ in \eqref{volumedensity} is determined by \eqref{ledgerformulaelse},
Lemma \ref{derivativesofq};
\begin{eqnarray}\label{ledger}
\left.\frac{d^2}{dr^2} \left(\frac{\Theta(r)}{r^{n-1}}\right)\right\vert_{r=0} = - \frac{1}{3}{\rm {Ric}}_u.
\end{eqnarray}
A slight computation shows us 
\begin{eqnarray}\label{densityexpansion}
\Theta(r) = k\, \frac{r^{n-1}}{2^{n-1}}\left(1 + \left( \frac{1}{3!} \frac{n-1}{2^2}+ \frac{1}{2}\frac{2Q-(n-1)}{2^2}\right) r^2 + \cdots 
\right)
\end{eqnarray}
so that the left hand side of \eqref{ledger} is given by $\displaystyle{\frac{k}{2^n}\left(Q-\frac{n-1}{3}\right)}$ and one gets $\displaystyle{k = - \frac{2^n}{3Q-(n-1)} {\rm {Ric}}_u}$.
To see $k= 2^{n-1}$ we use \eqref{meancurvatureexpression} together with the expansion of $\sigma(r)$ with respect to $r \rightarrow 0$, given at \eqref{sigmaledger} as 

\begin{eqnarray}\label{sigmaexpansion}
\sigma(r)
= \frac{n-1}{r} + \sum_{m =1}^{\infty} \frac{a_m}{m!}\, r^{m}.
\end{eqnarray}
Here
\begin{equation}\label{sigmaexpansioncoefficients}
a_m=\left\{
\begin{array}{cl}
\displaystyle\frac{B_{2\ell}}{2\ell}\left\{ (n-1)+(2Q- (n-1) )(2^{2\ell}-1)\right\},&m=2\ell-1,\medskip\\
0,&m=2\ell
\end{array}
\right.
\end{equation}
where $B_{2\ell}$ are the Bernoulli numbers (for definition of $B_{k}$ see \cite[\S 23]{AbSt}).
In particular,
\begin{eqnarray}\label{a3term}
a_1 = \dfrac{1}{2^2}\left\{\dfrac{n-1}{3}+(2Q-(n-1))\right\},\quad
a_3 = -\dfrac{3!}{2^4}\left\{\dfrac{n-1}{45}+ \dfrac{2Q-(n-1)}{3}\right\}.
\end{eqnarray}
  
\begin{lemma}\label{aone}
\begin{eqnarray}
a_1 = - \frac{1}{3} {\rm Ric}_u.
\end{eqnarray}
\end{lemma}
 
\begin{proof}
\eqref{densityexpansion} together with \eqref{a3term} tells us $\displaystyle{a_1= \lim_{r\rightarrow 0} \left(\log \frac{\Theta(r)}{r^{n-1}}\right)''}$. 
However, we find from Lemma \ref{derivativesofq}
\begin{align}
\lim_{r\rightarrow 0} \left(\log \frac{\Theta(r)}{r^{n-1}}\right)''
=& \lim_{r\rightarrow 0}\, \frac{(\frac{\Theta}{r^{n-1}})''(\frac{\Theta}{r^{n-1}})- (\frac{\Theta}{r^{n-1}})'(\frac{\Theta}{r^{n-1}})'}{(\frac{\Theta}{r^{n-1}})^2}\\
=& \lim_{r\rightarrow 0}\, \left(\frac{\Theta}{r^{n-1}}\right)'' = - \frac{1}{3} {\rm Ric}_u.
\end{align}
 \end{proof}
 
From this lemma together with \eqref{sigmaexpansioncoefficients} we have immediately
\begin{eqnarray}
 - \frac{2(n-1)}{3} + 2Q = - \frac{4}{3} {\rm Ric}_u,
\end{eqnarray}
namely, 
\begin{eqnarray}\label{qdelta}
Q = \frac{1}{3}\left(n-1 - 2 {\rm Ric}_u\right)
\end{eqnarray}
and hence $\displaystyle - \frac{2^n}{3Q- (n-1)} {\rm Ric}_u = 2^{n-1}$.
\end{proof}

Lemma \ref{aone} is also obtained from the Ledger's formulas.
Refer for the Ledger's formulas to \cite{Besse, MarcoSchueth}.
Let $\mathcal{C}(t):= t \mathcal{S}(t)$ be the endomorphism of $\gamma(t)^{\perp}$ by using the endomorphism $\mathcal{S}(t)$.
Then $\mathcal{C}(0)= \mathrm{Id}_{\gamma(0)^{\perp}}$ and $\mathcal{C}'(0)=0$.
So $\mathcal{C}(t)$ is smooth with respect to $t$. 
Then, from the Ledger's formulas $\mathcal{C}^{(\ell)}(0)$, the $\ell$-th derivative of $C(t)$ at $t=0$, $\ell> 1$ is expressed in principle in terms of curvature invariants.
The list of the exact expressions is appeared in \cite[p.162]{Besse} at least $2\leq \ell \leq 6$.
The mean curvature $\sigma(t)$, the trace of $\mathcal{S}(t)$, admits then the representation of
\begin{multline}\label{sigmaledger}
\sigma(t)= \frac{1}{t}\, {\rm Tr}\, \mathcal{C}(0) + {\rm Tr}\, \mathcal{C}'(0) + \frac{1}{2}\, {\rm Tr}\, \mathcal{C}''(0) t + \frac{1}{3!}\, {\rm Tr}\, \mathcal{C}'''(0) t^2\\
+ \cdots + \frac{1}{n!}\, {\rm Tr}\, \mathcal{C}^{(n)}(0)\ t^{n-1} + \cdots. 
\end{multline}
Therefore one obtains
\begin{multline}\label{sigmacurvature}
\sigma(t)= \frac{n-1}{t} - \frac{1}{3} \, {\rm Tr} R_u(0) \, t - \frac{1}{45} {\rm Tr}\, R^2_u(0)\,t^3 \\
- \frac{1}{3\cdot 7!} \big\{ 32\, {\rm Tr}\, R^3_u(0) - 9\, {\rm Tr}\, (R'_u(0))^2\big\} t^5 + \cdots,
\end{multline}
from which Lemma \ref{aone} is derived.

\medskip

Since each term $a_{\ell}$ of \eqref{sigmaexpansion} does not depend on $u\in S_oX$, we obtain from \eqref{sigmaledger} a countable many curvature identities of $(X,g)$.
That is, there exist constants $K$, $H$, and $L$ at least such that (\cite[6.46]{Besse});
\begin{equation*}\label{legderterms}
{\rm Ric}_u= {\rm Tr} R_u(0) = K, \qquad
{\rm Tr} R_u(0) R_u(0) = H,
\end{equation*}
\begin{equation*}
{\rm Tr}(32 R_u(0)\,R_u(0)\,R_u(0) - 9 R'_u(0)\,R'_u(0) ) = L
\end{equation*}
for any $u\in S_xX$, $x\in X$.
Here $R'_u(0)$ means the covariant derivative of $R_{\gamma'(t)}$ at $t=0$.
For these arguments and the above formulas refer to \cite{MarcoSchueth} and \cite[p. 162]{Besse}.
Now one obtains
\begin{align}
K=&\, {\rm Ric}_u = \frac{1}{2}((n-1)-3Q), \\ 
\label{h} H =&\, -\frac{1}{2^3}(7\, (n-1) - 15\, Q).
\end{align}
Similarly one has 
\begin{eqnarray}\label{el}L= 31\ (n-1) \ -\ 63\ Q.
\end{eqnarray}
Moreover, one calculates $C^{(7)}(0)$ as
\begin{align*}
C^{(7)}(0)
= & - \frac{21}{4} R_u^{(5)}(0) - \frac{35}{6}\left\{R_u(0) R_u^{(3)}(0) + R_u^{(3)}(0) R_u(0)\right\} \\
 & - \frac{77}{12}\left\{ R_u(0)R_u(0)R'_u(0)+ R'_u(0)R_u(0)R_u(0)\right\}\\
 & - \frac{35}{6} R_u(0)R'_u(0)R_u(0)
 - \frac{63}{4}\left\{ R_u'(0) R_u''(0) + R_u''(0)R'_u(0)\right\}.
\end{align*}
Since $\, {\rm Tr}\, C^{(7)}(0)/7!$ equals $\,a_6/6!$ of \eqref{sigmaexpansion} which vanishes for any $u$ from \eqref{sigmaexpansioncoefficients} for $(X,g)$ of hypergeometric type, one finds
\begin{eqnarray}
16\, {\rm Tr} R_u(0) R_u(0) R_u'(0) - 3\, {\rm Tr} R'_u(0) R''_u(0) = 0
\end{eqnarray}
for any $u$.
From this and the formula \eqref{el} ${\rm Tr}\, R_{\gamma'(t)} R_{\gamma'(t)} R_{\gamma'(t)}$ and ${\rm Tr}\, R_{\gamma'(t)}' R_{\gamma'(t)}'$ are both constant along the geodesic flow of the unit sphere bundle $SX$ over $(X,g)$.
For the detailed argument refer to \cite{ItohSatoh-3}. 
 
\begin{proof}[Proof of Theorem \ref{curvatureinequalities}] 
We are now at position to show Theorem \ref{curvatureinequalities}.
From \eqref{sigmacurvature} and \eqref{a3term} one has readily $\displaystyle {\rm Tr}\, R_u(0)^2 = \frac{1}{2^3}\, \{ 15 Q - 7(n-1)\}$ into which one inserts \eqref{qdelta} to get \eqref{riccitrace}.
The inequalities \eqref{inequalities} are obviously derived from \eqref{qinequality}. 
The proof for the equality cases follows from the fact that equality in the inequality $({\rm Tr}\, A)^2 \leq (n-1) ({\rm Tr}\, A^2)$ for a self-adjoint real endomorphism $A$ of $\mathbb{R}^{n-1}$ holds if and only if $A= \kappa\, {\rm Id}_{n-1}$ for $\kappa\in \mathbb{R}$.
\end{proof}

\begin{remark}
\eqref{volumedensity-1} of Theorem \ref{volumedensmeancurv} indicates 
\begin{eqnarray}\label{volumedensity expansion}
\label{asymptotical}
\Theta(r) = 2^{(n-1-2Q)}\, \exp Qr + o(1), \  r \rightarrow \infty.
\end{eqnarray}
Then a harmonic Hadamard manifold $(X,g)$ of $Q>0$ and of hypergeometric type is of purely exponential volume growth in the sense of G. Knieper \cite{K}.
So $(X,g)$ is Gromov hyperbolic and rank-one and the geodesic flow is Anosov.
We can relax the non-positivity of sectional curvature from our assumption as that $(X,g)$ be a non-compact, simply connected, complete harmonic manifold of $Q>0$.
However, under the non-positive sectional curvature condition we can assert that the principal curvatures of any horosphere associated to any $\theta\in\partial X$ are strictly positive everywhere.
This is because of rank one property of the harmonic Hadamard manifold $(X,g)$ of $Q>0$ and of hypergeometric type.
Refer to \cite{ItohKimParkSatoh}.
For an approach for determination of the volume density different from ours refer to \cite{N}. 
\end{remark}

\section{Volume entropy}

In this section we apply the Bishop comparison theorem to our Einstein manifold $(X^n, g)$ and prove Theorem \ref{volumeentropy-2}.

Let $\delta_g <0$ be the constant given by ${\rm Ric}_u= (n-1) \delta_g$.
Before showing Theorem \ref{volumeentropy-2}, we notice that $\displaystyle Q = \frac{n-1}{3}(1+ 2 \vert\delta_g\vert)$ from \eqref{qdelta}. 

\begin{proof}[Proof of Theorem \ref{volumeentropy-2}]
Since $(X,g)$ is Einstein, $\displaystyle {\rm Ric}_g(\cdot,\cdot) = (n-1) \delta_g\,g$.
Then the Bishop comparison theorem implies 
\begin{eqnarray}\label{comparison}
\Theta(r) \leq \left( {\bf S}_{\delta_g}(r)\right)^{n-1},\, r \geq 0,
\end{eqnarray}
where $\displaystyle {\bf S}_{\delta_g}(r) := \frac{1}{\sqrt{\vert\delta_g\vert}} \sinh(\sqrt{\vert\delta_g\vert}r)$ (refer to \cite[IV,\, Theorem 3.1, (2) b)]{Sakai}) and $\left({\bf S}_{\delta_g}(r)\right)^{n-1}$ represents the volume density of geodesic sphere of ${\Bbb R}H^n(\delta_g)$.
Since $\displaystyle{Q = \lim_{r\rightarrow\infty} \left(\log \Theta(r)\right)'}$, one observes from \eqref{comparison} that $Q$ satisfies

\begin{eqnarray}
Q \left(=\frac{n-1}{3}\left(2\vert\delta_g\vert+1\right)\right) \leq (n-1)\sqrt{\vert\delta_g\vert}
\end{eqnarray}
from which $\vert\delta_g\vert$ must satisfy
\begin{eqnarray}
\frac{1}{3}(2\vert\delta_g\vert+1) \leq \sqrt{\vert\delta_g\vert},
\end{eqnarray}
and hence $\displaystyle{\frac{1}{4}\leq \vert\delta_g\vert \leq 1}$. Then, $Q$ satisfies $\displaystyle{\frac{n-1}{2}\leq Q \leq n-1}$.

Now we suppose $Q = n-1$. Then $2Q-(n-1) = n-1$ and $\vert\delta_g\vert=1$. 
Thus $\displaystyle{\Theta(r) = 2^{n-1} \left(\frac{\sinh r}{2}\right)^{n-1} = \sinh^{n-1}r
}$ and one has equality in \eqref{comparison} for all $r>0$ and hence from the Bishop comparison theorem $(X,g)$ must be isometric to ${\Bbb R}H^n(-1)$. 

\medskip

Suppose $Q= (n-1)/2$. Then $\vert\delta_g\vert = 1/4$ and hence $\displaystyle \Theta(r) = 2^{n-1} \sinh^{n-1}\frac{r}{2}$.
On the other hand $\displaystyle {\bf S}_{\delta_g}(r) = \frac{1}{\sqrt{\vert\delta_g\vert}} \sinh\left(\sqrt{\vert\delta_g\vert}r\right) = 2 \sinh \frac{r}{2}$.
Thus one gets equality $\Theta(r) = \left({\bf S}_{\delta_g}\right)^{n-1}(r)$ in case of $Q= (n-1)/2$ so that from the Bishop comparison theorem again $(X^n, g)$ must be isometric to ${\Bbb R}H^n(-1/4)$. 
\end{proof}

It is easily shown that a real hyperbolic space ${\Bbb R}H^n(\delta)$ of constant sectional curvature $\delta$ is of hypergeometric type if and only if $\delta= -1$ or $-1/4$.

\section{The dual spherical Fourier transform ${\mathcal H}^-$} 

Let ${\mathcal H}^-$ be the dual spherical Fourier transform, defined at \eqref{definitiondualFT}, Section 1.
The aim of this section is to show the following.
 
\begin{proposition}\label{compactsupportthm}
The image of ${\mathcal H}^-$ is contained in $C^{\infty, \mathrm{rad}}_0(X)$.
More precisely, for any $h\in \mathcal{PW}({\Bbb C})^R_{\rm even}$, $R>0$,\, ${\mathcal H}^-(h)(r)$ is smooth and has a support contained in $[0,R]$. 
\end{proposition} 

\medskip

First we show that the integrand $\displaystyle{ h(\lambda) \varphi_{\lambda}(r) \vert {\bf c}(\lambda)\vert^{-2}}$ of \eqref{definitiondualFT} is integrable.
The integrability is in fact shown by means of the estimations \eqref{exponentialtype} for $\vert h(\lambda)\vert$ and (ii),\ Lemma \ref{estimationofspherical} for $\vert\varphi_{\lambda}(r)\vert$, respectively, together with the following for $\displaystyle{\frac{1}{\vert{\bf c}(\lambda)\vert} }$; for any $\varepsilon \geq 0$ there exists $K_1 > 0$ such that
\begin{eqnarray}\label{estimatec}
\frac{1}{\vert {\bf c}(\lambda)\vert} \leq K_1 (1+ \vert \lambda\vert)^{(n-1)/2}
\end{eqnarray}
for any $\lambda= \xi + i\eta$ with $\eta \geq - \varepsilon\vert\xi\vert$ (\cite[Theorem 2(iii)]{Flensted}).
 
\medskip

We will show next that $\mathcal{H}^-(h)(r)$ has support in $[0,R]$.
Before showing this, we apply the connection formula for hypergeometric functions which is given at \eqref{connectionformula} to $\varphi_{\lambda}(r)$ as
\begin{eqnarray}
\varphi_{\lambda}(r)= {\bf c}(\lambda)\Phi_{\mu}(t)+ {\bf c}(-\lambda)\Phi_{-\mu}(t),\ t=\frac{r}{2},\, \mu= 2\lambda. 
\end{eqnarray}
Here ${\bf c}(\lambda)$ is the Harish-Chandra $c$-function given at \eqref{harishChandrac} so that from \eqref{c} ${\bf c}(\lambda)=(2\sqrt{\pi})^{-1}\Gamma(n/2) c_{\alpha,\beta}(2\lambda)$.
Moreover $\Phi_{\mu}(t) = \Phi_{\mu}^{(\alpha,\beta)}(t)$, $\alpha= n/2\,-\ 1$, $\beta= Q - n/2$,\, is the Jacobi function of second kind, which is given at \eqref{jacobiseconkind}. 

For this purpose we set ${\breve{h}}(r):= \mathcal{H}^-(h)(r)$ and show the following.
\begin{lemma}\label{hormandertrick}
There exists a constant $K_2>0$ such that for any fixed $\eta \geq 0$ it holds for any $r>0$, $\vert \breve{h}(r)\vert \leq K_2\, e^{(R-r)\eta}$.
\end{lemma}

This lemma implies that ${\breve{h}}(r)= 0$ whenever $r>R$ so that ${\rm supp}\, ({\breve{h}}) \subset [0,R]$.
 
\begin{proof}
Since $h(\lambda)$ and $\varphi_{\lambda}(r)$ are even functions of $\lambda$, we write ${\breve{h}}(r)$ as
\begin{eqnarray}\label{integralexpression}
{\breve{h}}(r) = d_g \int_{-\infty}^{\infty} h(\lambda) \frac{\Phi_{\mu}(t)}{{\bf c}(-\lambda)}\, d\lambda,\, \mu= 2\lambda, t= \frac{r}{2}
\end{eqnarray}
and then express \eqref{integralexpression} as the line integral
\begin{eqnarray}\label{lineintegral1}
{\breve{h}}(r)= d_g\, \int_{C_1}\, h(\lambda) \frac{\Phi_{\mu}(t)}{{\bf c}(-\lambda)}\, d\lambda
\end{eqnarray}
along the real axis, namely along the path $C_1\,:\, \mathbb{R}\ni\xi \mapsto \xi\in\mathbb{C}$.
The functions $h(\lambda)$ and $\Phi_{\mu}(t)$, $\mu=2\lambda$ are holomorphic in $\{\lambda=\xi+i\, \eta; \eta\geq 0\}$.
On the other hand, one can observe from \eqref{c} that $\displaystyle{\frac{1}{{\bf c}(-\lambda)}=\frac{2\sqrt{\pi}}{\Gamma(n/2)} \frac{1}{c_{\alpha,\beta}(-\mu)}}$ is a meromorphic function of $\lambda$ whose poles are located in the set $A\cup B$, $A= \{-i (Q/2 + \ell); \ell\in {\Bbb N}\}$, $B=\{-i((n-Q)/2+ \ell);\, \ell\in {\Bbb N}\}$.
From Theorem 1.7 the volume entropy $Q$ satisfies $n- Q \geq 1$ so that  $\displaystyle{\frac{1}{{\bf c}(-\lambda)}}$ is holomorphic in $\{\lambda=\xi+i\eta; \eta\geq 0\}$.
Therefore, we can apply the Cauchy's integral theorem.
Then \eqref{lineintegral1} equals the line integral along a path $C_2=C_{2,\eta};\, {\Bbb R}\ni \xi\mapsto \xi+ i\eta\in {\Bbb C}$ parametrized by a fixed $\eta> 0$;
\begin{eqnarray}\label{integral-1}
d_g\ \int_{C_1} h(\lambda) \frac{\Phi_{\mu}(t)}{\bf{c}(-\lambda)} d \lambda = d_g\ \int_{C_{2,\eta}} h(\lambda) \frac{\Phi_{\mu}(t)}{\bf{c}(-\lambda)} d \lambda,\ \mu=2\lambda,\, t= \frac{r}{2} .
\end{eqnarray}
In fact, \eqref{integral-1} is derived by considering the line integral along the contour 
\begin{eqnarray}
C = C_1\vert_{[-r,r]} + C_{3,r} - C_2\vert_{[-r,r]} - C_{4,r}.
\end{eqnarray}
Here $r>0$ is sufficiently large and $C_j\vert_{[-r,r]}$, $j=1,2$ is the restriction of $C_j$ to $[-r,r]$ and $C_{3,r}$ and $C_{4,r}$ are the paths defined by $[0,\eta]\ni s \mapsto \pm\ r + i s$, respectively.
Since
\begin{eqnarray}
\int_C h(\lambda) \frac{\Phi_{2\lambda}(t)}{\bf{c}(-\lambda)} d \lambda =0
\end{eqnarray}
and the integrals along $C_{3,r}$ and $C_{4,r}$ tend to zero from Lemma \ref{asymptoticallemma} below, one obtains \eqref{integral-1} by letting $r \rightarrow +\infty$.

Now we take $t> 0$ and $\eta > 0$ as arbitrary fixed numbers. Then from Lemma \ref{lemma}
\begin{align}
\vert{\breve{h}}(2t)\vert
\leq & d_g \int_{-\infty}^{\infty} \vert h(\xi+i\eta)\vert \left\vert\frac{1}{{\bf{c}}(-\xi+i\eta)}\right\vert \vert\Phi_{2(\xi+i\eta)}(t)\vert\,d\xi \notag\\
\leq & K e^{(R-2t)\eta} \int_{-\infty}^{\infty} \frac{(1+\vert\xi+i\eta\vert)^{\frac{n-1}{2}}}{(1+\vert\xi+i\eta\vert)^N }d\xi
\end{align}
in which we choose $N$ such as $N > 2 + \frac{n-1}{2}$ so Lemma \ref{hormandertrick} is shown. 
\end{proof}

\begin{lemma}\label{asymptoticallemma}
$\Phi_{\mu}(t)$ can be described as
\begin{eqnarray}
\Phi_{\mu}(t) = e^{(i\mu- Q)t}\left(1+ e^{-2t} \Psi(\mu,t)\right) 
\end{eqnarray}
with respect to a certain function $\Psi(\mu,t)$ of $(\mu,t)$ which satisfies the following; for any $s > 0$, any $\varepsilon > 0$ and any $m\in {\Bbb N}$ there exists a $K_m > 0$ such that for all $\mu=\xi+ i\eta \in {\Bbb C}$ with $\eta \geq - \varepsilon \vert \xi\vert$ and for all $t\in [s,\infty)$
\begin{eqnarray}
\left\vert \frac{d^m}{dt^m} \Psi(\mu,t)\right\vert \leq K_m.
\end{eqnarray}
\end{lemma}
See \cite[Theorem 2]{Flensted} for this lemma.
The next lemma is obtained from \eqref{exponentialtype}, \eqref{estimatec} and Lemma \ref{asymptoticallemma}.
\begin{lemma}\label{lemma}
There exists a constant $K > 0$ such that the following holds for any fixed $t> 0$, any $N\in{\Bbb N}$ and any fixed $\eta > 0$ 
\begin{eqnarray}\label{estimate}
\left\vert h(\xi+i\eta) \frac{\Phi_{2(\xi+i\eta)}(t)}{{\bf{c}}(-(\xi+i\eta))}\right\vert \leq K e^{(R-2t)\eta} (1+ \vert\xi+i\eta\vert)^{(\frac{n-1}{2}-N)}\ \forall \xi.
\end{eqnarray}
\end{lemma}

The smoothness of ${\breve{h}}(r)$ is shown from Lemma \ref{asymptoticallemma} by commutability of differentiation by $r$ and integration by $\lambda$.
Proposition \ref{compactsupportthm} is thus verified.

\section{Convolution Rule}

Let $f$, $h\in {\mathcal C}_0^{\infty, \mathrm{rad}}(X)$ be any radial smooth functions on $X$ of compact support. The convolution of $f$ with $h$ is defined by
\begin{eqnarray}
f \ast h(x) :=\int_{y\in X} f(d(o,y))\, h(d(y,x))\,dv_g(y),\, x\in X.
\end{eqnarray} 
The following is verified by Z. Szab$\acute{\rm o}$.
\begin{theorem}[\cite{Sz}]
Let $(X^n,g)$ be a simply connected, non-compact complete harmonic manifold. Then $f \ast h\in C^{\infty, \mathrm{rad}}_0(X)$ for any $f,\, h\in C^{\infty, \mathrm{rad}}_0(X)$.
\end{theorem}

\begin{theorem}\label{convolutionrule}
Let $f$, $h\in {\mathcal C}_0^{\infty, \mathrm{rad}}(X)$.
Then
\begin{eqnarray}
{\widehat{(f\ast h)}}(\lambda)
= \hat{f}(\lambda) \cdot {\hat{h}}(\lambda). 
\end{eqnarray}
\end{theorem}

\begin{proof}
Set $F(r(x)) := \left(f\ast h\right)(r(x))$ for simplicity.
Then the spherical Fourier transform of $F$ is given by \eqref{sphericalfourier} as
\begin{eqnarray}\label{sphericalfourier-2}
{\hat{F}}(\lambda) = \int_{x\in X} F(r(x))\, \varphi_{\lambda}(r(x))\,dv_g(x).
\end{eqnarray}
Since the function $F$ is radial and $\varphi_{\lambda}(r(x))$ is the average of the $\lambda$-Poisson kernel $P_{\lambda}(x,\theta)$, we can write \eqref{sphericalfourier-2} in terms of the $\lambda$-Poisson kernel as
\begin{eqnarray}\label{sphericalfourier-3}
{\hat{F}}(\lambda) = \int_{x\in X} F(d(x,o)))\, P_{\lambda}(x,\theta)\,dv_g(x)
\end{eqnarray}
for a $\theta\in\partial X$.
Therefore
\begin{align}\label{computation-1}
{\hat{F}}(\lambda)
=& \int_{x\in X}\,\left(\int_{y\in X} f(d(o,y)) h(d(y,x))\,dv_g(y)\right) P_{\lambda}(x,\theta)\,dv_g(x) \nonumber\\ 
=& \int_{y\in X} f(d(o,y))\,\left(\int_{x\in X} h(d(y,x)) P_{\lambda}(x,\theta)\,dv_g(x) \right)\,dv_g(y).
\end{align}
Here the $\lambda$-Poisson kernel has from Definition \ref{lambdapoisson} the form represented by the Busemann function $b_{\theta}(x)$ which is associated with the geodesic $\gamma$ of $\gamma(0) = o$, $\theta=[\gamma]$.
For each $y\in X$ let $\gamma_y$ be a geodesic, asymptotically equivalent to $\gamma$ and $\gamma_y(0) = y$.
Then $\theta$ is also represented by $\gamma_y$.
Therefore, for the Busemann function $b_{\gamma_y}(x)$ associated with the geodesic $\gamma_y$ we see from \eqref{coclosedcondition} $b_{\gamma}(x) - b_{\gamma_y}(x) \equiv b_{\gamma}(y)$ for any $x\in X$, since $b_{\gamma_y}(y) = 0$.
This implies that for any $x$
\begin{align}\label{cocyclepoisson}
P_{\lambda}(x,\theta)
=& \exp\left\{\left(-\frac{Q}{2}+i\lambda\right)b_{\gamma_y}(x)\right\}\, \exp\left\{\left(-\frac{Q}{2}+i\lambda\right) b_{\gamma}(y)\right\}\nonumber\\
=& P_{\lambda}^y(x,\theta)\, P_{\lambda}(y,\theta).
\end{align}
Here $P_{\lambda}^y(x,\theta)$ denotes the $\lambda$-Poisson kernel associated with $b_{\gamma_y}(\cdot)$ so that $P_{\lambda}^y(y,\theta) = 1$ at the point $y$.
Notice $P_{\lambda}(x,\theta) = P_{\lambda}^o(x,\theta)$, $\forall x\in X$ where $o$ is the reference point of $X$. 

Using \eqref{cocyclepoisson}, we write \eqref{computation-1} as
\begin{equation}\label{computation-3}
{\hat{F}}(\lambda) =\int_{y\in X} f(d(o,y)) \left(\int_{x\in X} h(d(y,x)) P_{\lambda}^y(x,\theta)\,dv_g(x) \right) P_{\lambda}(y,\theta)\,dv_g(y).
\end{equation}
The integral of the parentheses $\left(\dots\right)$ can be written as $\displaystyle \hat{h}(\lambda)$, the spherical Fourier transform of the radial function $h$.
In fact one writes from \eqref{avrgop} this integral around the point $y$ as
\begin{align}
\int_{x\in X} h(d(y,x)) P_{\lambda}^y(x,\theta)\,dv_g(x)
=& \int_0^{\infty} h(r) dr \int_{x\in S(y;r)} P_{\lambda}^y(x,\theta)\,dv_{S(y;r)} \nonumber\\ 
=& \omega_{n-1} \int_0^{\infty} h(r) \varphi_{\lambda}(r) \Theta(r)\,dr
= {\hat{h}}(\lambda).
\end{align}
The last equality is from Definition \ref{definition-1}.
Thus \eqref{computation-3} becomes
\begin{equation}
{\hat{F}}(\lambda)
= {\hat{h}}(\lambda) \int_{y\in X} f(d(o,y)) P_{\lambda}(y,\theta)\,dv_g(y)
= {\hat{h}}(\lambda) \hat{f}(\lambda)
\end{equation}
from which our theorem is verified.
\end{proof}

\begin{remark}
In \cite[Theorem 3.12]{PS},
Peyerimhoff and Samiou obtained the convolution rule for radial distributions on a non-compact simply connected, complete harmonic manifold.
Our proof is based on the asymptotical property of Busemann function on a Hadamard manifold.
\end{remark}

\section{A proof of Plancherel Theorem}

\begin{proof}[Proof of Theorem \ref{plancherel-1}]
Set
\begin{eqnarray}\label{fr}
F(r(x)) := (f\ast {\overline h})(r(x))
\end{eqnarray}
Then, from the convolution rule derived in Theorem \ref{convolutionrule}
\begin{equation}
{\hat{F}}(\lambda) = \hat{f}(\lambda) {\overline {\hat{h}}}(\lambda),
\end{equation}
where ${\overline {\hat{h}}}(\lambda) = {\hat{\overline{h}}}(\lambda)$.
Since $r(x)=0$ if and only if $x=o$, we put $r(x)=0$ in \eqref{fr} so that
\begin{equation}
F(0) = (f\ast {\overline h})(0) = \int_{y\in X} f(d(o,y)) {\overline h}(d(y,o))\,dv_g(y) 
\end{equation}
which is the inner product $\langle f, h \rangle_{(L_2, \omega_{n-1} \Theta dr)}$. 
On the other hand, by the inversion formula appeared at Theorem \ref{inversion-1} we have
\begin{eqnarray}
F(r(x)) = \int_{0}^{\infty} {\hat{F}}(\lambda) \varphi_{\lambda}(r(x)) \frac{2\ d_g}{\vert{\bf c}(\lambda)\vert^2}\,d\lambda
\end{eqnarray}
and hence
\begin{eqnarray}\label{flambda}
F(r(x)) = \int_{-\infty}^{\infty} {\hat{F}}(\lambda) \varphi_{\lambda}(r(x)) \frac{d_g}{\vert{\bf c}(\lambda)\vert^2}\,d\lambda,
\end{eqnarray}
since ${\hat{F}}(\lambda)$ is an even function of $\lambda$.
Here, we put $r(x) = 0$ into \eqref{flambda} and pay attention to $\varphi_{\lambda}(0) = 1$ to obtain
\begin{eqnarray}
F(0)
= \int_{-\infty}^{\infty} {\hat{F}}(\lambda) \frac{d_g}{\vert{\bf c}(\lambda)\vert^2}\,d\lambda
= \int_{-\infty}^{\infty} \hat{f}(\lambda)\, {\overline {\hat{h}}}(\lambda) \frac{d_g}{\vert{\bf c}(\lambda)\vert^2}\,d\lambda
\end{eqnarray}
which is indeed the inner product $\langle \hat{f}, {\hat{h}}\rangle_{(L_2, d_g\vert{\bf c}(\lambda)\vert^{-2}d\lambda)}$ of $\hat{f}$ and ${\hat{h}}$.
\end{proof}

\begin{remark}
The Plancherel Theorem for the Jacobi transform of order $(\alpha,\beta)$,\, $\alpha, \beta\in {\Bbb R}$, $\vert\beta\vert< \alpha + 1$ is verified in \cite[p.156]{Ko-x}.
See also \cite[Prop. 3]{Flensted}.
\end{remark}

\section{Appendix}

\begin{proposition}
Suppose there exists a variable transformation $z = z(r)$ which takes the equation \eqref{eigenequation} with an arbitrary real parameter $\lambda$ of $\lambda^2+ Q^2/4 >0$ associated to $\varphi_{\lambda}(r)$ into the equation \eqref{hypergeometricequation} associated to a certain function $f(z)$. 

Then it is concluded that $c= n/2$, $ab > 0$ and further $a = \dfrac{1}{\ell} \left(\dfrac{Q}{2} \pm i \lambda\right)$ and $b = \dfrac{1}{\ell} \left(\dfrac{Q}{2} \mp i \lambda\right)$,
where $\ell = \sqrt{\dfrac{Q^2/4+\lambda^2}{ab}}> 0$, in \eqref{hypergeometricequation} and moreover the transformation $z= z(r)$ must be $z = - \sinh^2 \ell r/2$ so that the exact form of $\varphi_{\lambda}$ is
\begin{eqnarray}
\varphi_{\lambda}(r) = F\left(\frac{1}{\ell} \left(\frac{Q}{2} \mp i \lambda\right), \frac{1}{\ell} \left(\frac{Q}{2} \pm i \lambda\right), \frac{n}{2}; z\right), \, z= - \sinh^2 \frac{\ell r}{2}.
\end{eqnarray}
\end{proposition}

\begin{proof}
Put $\varphi(r)=\varphi_{\lambda}(r)$ for simplicity and suppose that $\varphi(r)$ is converted into $f(z)$ by $z=z(r)$ as $f(z(r))= \varphi(r)$.
Then, similarly as in the proof of Theorem \ref{hypergeometrictype-1} (see \cite{ItohSatoh-2})
\begin{eqnarray}
\frac{d\varphi}{d r} =\frac{d f}{d z}\, \frac{d z}{d r} 
\end{eqnarray}
and 
\begin{eqnarray}
\frac{d^2\varphi}{d r^2} =\frac{d^2 f}{d z^2}\left(\frac{d z}{d r}\right)^2 + \frac{d f}{d z}\, \frac{d^2 z}{d r^2} 
\end{eqnarray}
so that
\begin{multline}
\left(\frac{d^2}{dr^2}+ \sigma(r)\frac{d}{dr}\right) \varphi + \left(\frac{Q^2}{4}+\lambda^2\right) \varphi 
\\
= \frac{d^2 f}{d z^2}\, \left(\frac{d z}{d r}\right)^2 +  \left(\frac{d^2 z}{d r^2} + \sigma(r)\, \frac{d z}{d r} \right)\frac{d f}{d z}\, + \left(\frac{Q^2}{4}+\lambda^2\right)\, f = 0.
\end{multline}
This equation must turn out to be the equation \eqref{hypergeometricequation}.
Hence, since $Q^2/4+\lambda^2 \not= 0$, it follows $ab \not=0$.
Therefore we get
\begin{align}\label{firsteqn}
\left(\frac{d z}{d r}\right)^2 
=& - \frac{\left(\frac{Q^2}{4}+\lambda^2\right)}{ab}\, z(1-z),\\
\label{secondeqn}
\frac{d^2 z}{d r^2} + \sigma(r)\, \frac{d z}{d r} 
=& - \frac{\left(\frac{Q^2}{4}+\lambda^2\right)}{ab}\,\{c- (a+b+1)z\}.
\end{align}
From the following considerations together with Lemma \ref{z->1-z} below,
we find that there are essentially two solutions to \eqref{firsteqn};
\begin{enumerate}
\item when $ab<0$, $z=\sin^2(\ell r/2+C)$,
\item when $ab>0$, $z=-\sinh^2(\ell r/2+C)$,
\end{enumerate}
where $C$ is an arbitrary constant, and that the case (i) is not appropriate and the case (ii) with $C=0$ is appropriate for our argument.

\medskip

\noindent{\bf Case (i)}\, When $ab < 0$, one has from \eqref{firsteqn} $z(1-z) \geq 0$ and hence $0 \leq z \leq 1$.
Then, \eqref{firsteqn} is solved by setting $z = \sin^2 \ell t$ for a function $t=t(r)$, where $\ell = \sqrt{-(Q^2/4+ \lambda^2)/ab}$ and inserting it to have $\left(\frac{d z}{d r}\right)^2 = 4\sin^2 \ell t \cos^2 \ell t \cdot \ell^2 \left(\frac{dt}{dr}\right)^2$ and $\ell^2 z(1-z) = \ell^2 \sin^2\ell t \cos^2 \ell t$ and then $\left(\frac{dt}{dr}\right)^2 = 1/4$, namely one may put $t = r/2+C/\ell$ and $z= \sin^2( \ell r/2+C)$.
The left hand of \eqref{secondeqn} is then written
\begin{eqnarray}\label{lefthand}
\frac{d^2 z}{d r^2} + \sigma(r)\, \frac{d z}{d r}
= \frac{\ell}{2}\left(\ell \cos (\ell r+2C)+ \sigma(r)\sin (\ell r+2C)\right)
\end{eqnarray}
and the right hand is
\begin{eqnarray}\label{righthand}
\ell^2\{c-(a+b+1)z\} = \ell^2 \left\{ c - \frac{1}{2}(a+b+1) + \frac{1}{2}(a+b+1) \cos (\ell r+2C)\right\}.
\end{eqnarray}
When $r \rightarrow 0$, from \eqref{meancurvature} $\sigma(r) = \frac{n-1}{r} + o(1)$, so if $2C\notin \mathbb{Z}$ then \eqref{secondeqn} converges to a finite value, but \eqref{lefthand} diverges.
Hence we find that $C$ must be $0$.
If $C=0$, then \eqref{lefthand} tends to $\ell/2(\ell + \ell(n-1)) = n \ell^2/2$ and \eqref{righthand} goes to $\ell^2 c$ as $r \rightarrow 0$ so that $c = n/2$.
Since \eqref{lefthand} equals \eqref{righthand}, one has
\begin{align}
\frac{\sigma(r)}{2\ell}\, \sin \ell r
=& c - \frac{1}{2} (a+b+1) + \frac{1}{2}(a+b)\cos \ell r \nonumber \\
=& \label{equation-x} \frac{n-1}{2} + \frac{1}{2} (a+b)(\cos \ell r - 1).
\end{align} 
If $a+b >0$, then choose a small $\varepsilon > 0$ such that for any $\ell r \in (2\pi-\varepsilon, 2\pi)$ it holds 
\begin{eqnarray}
1- \cos \ell r < \frac{1}{2} \, \frac{1}{a+b}\quad\mbox{and}\quad\sin \ell r < 0.
\end{eqnarray}
Then, the left hand of \eqref{equation-x} is negative, while the left hand is positive. This is a contradiction.

On the other hand, if $a+b \leq 0$, then choose $r > 0$ such that $\cos \ell r < 1/2$ and $\sin \ell r < 0$ to have $\sigma(r) < 0$, a contradiction, since one has $\displaystyle{\frac{1}{2}(a+b)(\cos \ell r -1) >0
}$, while from \cite[Lemma 5.5]{ItohSatoh-2}\, $\sigma(r) \rightarrow Q>0$, as $r\rightarrow \infty$.

\medskip

\noindent{\bf Case (ii)}\, When $a b > 0$, $z(1-z) \leq 0$, so either $z \leq 0$ or $z \geq 1$.
We get $z= -\sinh^2 (\ell r/2+C)$ for the case of $z \leq 0$ or $z = \cosh^2 (\ell r/2+C)$ for $z \geq 1$ in a similar manner as case (i). Here $\ell = \sqrt{(Q^2/4+ \lambda^2)/ab}>0$.
From Lemma \ref{z->1-z}, even if we consider only the former, generality is not lost.
Moreover, we find that $C$ must be $0$ in a similar manner as case (i).

The left hand of \eqref{secondeqn} is written
\begin{eqnarray}\label{lefthand-2}
\frac{d^2 z}{d r^2} + \sigma(r)\, \frac{d z}{d r} = - \frac{\ell}{2}\left(\ell \cosh \ell r+ \sigma(r)\sinh \ell r\right)
\end{eqnarray}
and the right hand is
\begin{eqnarray}\label{righthand-2}
-\ell^2\{c-(a+b+1)z\} = -\ell^2 \left\{ c + (a+b+1) \sinh^2 \frac{\ell r}{2}\right\}
\end{eqnarray}
which tends to $-\ell^2 c$ as $r\rightarrow 0$.
On the other hand \eqref{lefthand-2} tends to $- \ell^2 n/2$. Thus $c = n/2$. 
Since $\sigma(r) \rightarrow Q$, $r\rightarrow\infty$,
we obtain $(a + b) \ell = Q$ from \eqref{lefthand-2} and \eqref{righthand-2}.
Therefore $a = (Q/2 + i \lambda)/\ell$ and $b = (Q/2 - i \lambda)/\ell$. 
\end{proof}

\begin{lemma}\label{z->1-z}
If $u(z)$ is a solution to \eqref{hypergeometricequation},
then $v(z):=u(1-z)$ is a solution to the hypergeometric differential equation of another type;
\begin{equation*}
z(1-z)\,f''(z)+\{c_1-(a_1+b_1+1)z\}\,f'(z)-a_1b_1\,f(z)=0,
\end{equation*}
$a_1=a$, $b_1=b$, $c_1=a+b+1-c$.
\end{lemma}

\end{document}